\documentclass[11pt]{amsart}

\usepackage[left=1.0in,top=0.95in,right=1.0in,bottom=0.95in]{geometry}

\usepackage{amssymb}
\usepackage{amsthm}
\usepackage{epsfig}
\usepackage[T1]{fontenc}
\usepackage{color}
\usepackage{verbatim}
\usepackage{appendix}
\usepackage{mathtools}
\usepackage{tikz}
\usepackage{enumitem}

\usepackage{amsmath}
\usepackage{hyperref}
\usepackage[active]{srcltx}
\usepackage{xypic}
\usepackage{amscd}
\usepackage{color}
\usepackage{xcolor}
\usepackage[normalem]{ulem}

\theoremstyle{plain}
\newtheorem{thm}{Theorem}[section]

\newtheorem{lemma}[thm]{Lemma}
\newtheorem{prop}[thm]{Proposition}
\newtheorem{defn}[thm]{Definition}
\newtheorem{claim}[thm]{Claim}

\newtheorem{conjecture}[thm]{Conjecture}

\newtheorem{theoremintro}{Theorem}

\newtheorem{questionintro}{Question}

\newtheorem{corintro}{Corollary}

\theoremstyle{definition}
\newtheorem{remark}[thm]{Remark}

\newtheorem{example}[thm]{Example}

\DeclareMathOperator{\im}{im}

\DeclareMathOperator{\Pic}{Pic}

\newcommand{\g}[2]{g^{#1}_{#2}}
\newcommand{\BN}[3]{\mathcal{M}^{#2}_{#1,#3}}
\newcommand{\KK}[3]{\mathcal{K}^{#2}_{#1,#3}}
\newcommand{\abs}[1]{|#1|}

\newcommand{\floor}[1]{\left\lfloor #1 \right\rfloor}
\newcommand{\ceil}[1]{\left\lceil #1 \right\rceil}
\newcommand{\dmax}{d_{max}}

\usepackage{mathrsfs}  
\DeclareMathOperator{\ext}{\mathscr{E}\!\mathit{xt}}
\newcommand{\sheaf}[1]{\mathscr{#1}}
\def\A{\sheaf{A}}
\def\B{\sheaf{B}}
\def\E{\sheaf{E}}
\def\F{\sheaf{F}}
\def\G{\sheaf{G}}
\def\I{\sheaf{I}}
\def\L{\sheaf{L}}
\def\O{\sheaf{O}}
\def\T{\sheaf{T}}

\def\coker{\operatorname{coker}}

\def\min{\operatorname{min}}

\def\max{\operatorname{max}}
\def\length{\operatorname{length}}
\def\c1{\operatorname{c_1}}
\def\c2{\operatorname{c_2}}

\def\rk{\operatorname{rk}}

\def\c{\mathfrak{P}}
\def\CC{{\mathbb C}}
\def\ZZ{{\mathbb Z}}

\def\PP{{\mathbb P}}

\def\M{{\mathcal M}}

\def\K{{\mathcal K}}

\def\K{{\mathcal K}}

\def\P{{\mathcal P}}

\def\cong{\simeq}

\def\+{\oplus}                   
\def\*{\otimes}                  

\def\Pic{\operatorname{Pic}}

\begin{document}  
\thispagestyle{empty}
\vspace*{-.5cm}
	
\title{Distinguishing Brill--Noether loci}
	
\author[A.~Auel]{Asher Auel}
\address{A.~Auel, Department of Mathematics\\%
		Dartmouth College\\%
		Kemeny Hall\\%
		Hanover, NH 03755}
\email{asher.auel@dartmouth.edu}
	
\author[R.~Haburcak]{Richard Haburcak}
\address{R.~Haburcak, Department of Mathematics\\%
		Dartmouth College\\%
		Kemeny Hall\\%
		Hanover, NH 03755}
\address{Max-Planck-Institut für Mathematik\\%
		Vivatsgasse 7\\%
		53111 Bonn\\%
		Germany}
\email{richard.haburcak@dartmouth.edu}

\author[A.~L.~Knutsen]{Andreas Leopold Knutsen}
\address{A.~L.~Knutsen, Department of Mathematics, University of Bergen,
		Postboks 7800,
		5020 Bergen, Norway}
	\email{andreas.knutsen@math.uib.no}

\begin{abstract}
We construct curves carrying certain special linear series and not
others, showing many non-containments between Brill--Noether loci in
the moduli space of curves. In particular, we prove the Maximal Brill--Noether Loci conjecture in
full generality.  
\vspace{-3em}
\end{abstract}

\maketitle
	
\section*{Introduction} 
\label{sec:intro}
	
Whereas classical Brill--Noether theory studies linear systems on
general algebraic curves, \emph{refined Brill--Noether theory} aims to
characterize linear systems on special curves. A degree $d$ linear
system of dimension $r$, called a $\g{r}{d}$, corresponds to a
non-degenerate morphism $C\to \mathbb{P}^r$ of degree $d$ when it is base point free. The
Brill--Noether--Petri theorem \cite{Gies_BN,griffiths_harris,Laz}
states that a general smooth curve $C$ of genus $g$ admits a
$\g{r}{d}$ if and only if the \emph{Brill--Noether number}
\[
\rho(g,r,d) = g - (r+1)(g-d+r)
\]
is non-negative. The last few years have seen a flurry of results
concerning refined Brill--Noether theory of curves in a fixed
\emph{Brill--Noether locus}
\[
\BN{g}{r}{d} = \{C\in \M_g ~:~ C \text{ admits a }\g{r}{d}\}
\] 
when $\rho(g,r,d)<0$. In particular, major advances in refined
Brill--Noether theory for curves of fixed gonality, i.e., when $r=1$,
have been made in
\cite{cook-powell_jensen,jensen_ranganathan,larson_refined_BN_Hurwitz,larson_larson_vogt_2020global,pflueger}.
	
Part of a refined Brill--Noether theory concerns the question of
whether a ``general'' curve with a $g^r_d$ carries another
$g^{r'}_{d'}$, which can be reinterpreted in terms of containments of
Brill--Noether loci.  For example, Clifford's theorem implies that
$\BN{g}{r}{2r} \subset \BN{g}{1}{2}$ for every $r \geq 1$.  By adding
base points and 
subtracting non-base points, one obtains the \emph{trivial
containments} $\M^r_{g,d} \subset \M^r_{g,d+1}$ and $\M^r_{g,d}
\subset \M^{r-1}_{g,d-1}$ between Brill--Noether loci. The
\emph{expected maximal} Brill--Noether loci are those that do not
admit further trivial containments, see
Definition~\ref{defn:expected_maximal} for the precise definition.
	
The interaction between various Brill--Noether loci is useful in the study
of the birational geometry of $\M_g$, see~\cite{Farkas2000,harris_mumford}. 
When $\rho=-1$, the
Brill--Noether loci are irreducible divisors, which
have been studied by Harris, Mumford, Eisenbud, and Farkas
\cite{EisenbudHarrisBN-1,eisenbud_harris_Kodaira,Farkas2000,Farkas_2001,harris_mumford}. A
crucial ingredient in the study of the Kodaira dimension of $\M_{23}$
was the maximality of the Brill--Noether divisors.
		
Inspired by the lifting of line bundles on curves on K3 surfaces, the Donagi--Morrison conjecture, building on work of Farkas and Lelli-Chiesa~\cite{Farkas2000,Farkas_2001,lelli_rank2}, and classical results in Brill--Noether theory, the first two authors~\cite{AH} formulated a conjecture stating that the expected maximal
Brill--Noether loci are exactly the maximal ones with respect to
containment, except precisely in genus $g=7,8,9$, where there are
exceptional cases, see Conjecture~\ref{Conj Max BN loci}, and for details on the exceptional genera see Example~\ref{ex: unex cont} below. In other
words, given any two expected maximal Brill--Noether loci
$\BN{g}{r}{d}$ and $\BN{g}{r'}{d'}$, there is a curve $C\in \M_g$
admitting a $\g{r}{d}$ but not a $\g{r'}{d'}$. An appealing aspect of
the conjecture is a numerical characterization of the maximal elements
of the Brill--Noether stratification of $\M_g$. There has been recent
progress on this conjecture, referred to as the Maximal Brill--Noether
Loci conjecture, in work of many authors
\cite{AHL,bud2024brillnoether,BH_2024_BN_deg,CHOI2022,CHOI20141458,lelli_rank2,bigas2023brillnoether},
which together show the conjecture holds in genus $g\leq 23$ and for
infinitely many values of $g$.

In this paper, we settle the conjecture in the affirmative.  In fact,
we prove something stronger.
	
\begin{theoremintro}
\label{thm:conj_intro}
In every genus $g \geq 3$, with $g \neq 7,8,9$, the expected maximal
Brill--Noether loci $\BN{g}{r}{d}$ are all distinct and are the
Brill--Noether loci that are maximal with respect to containment.
More precisely, $\BN{g}{r}{d}$ has a component on which the general
curve does not admit any $g^{r'}_{d'}$ with $ d' \leq g-1$,
$\rho(g,r',d')<0$, and $(r',d') \neq (r,d)$, unless $(g,r,d)=(7, 2, 6), (8,1,4), (9,2,7)$. 
\end{theoremintro}

We note that this theorem is stronger than the Maximal Brill--Noether Loci conjecture, as it shows the existence of a component in each expected maximal $\BN{g}{r}{d}$ demonstrating all non-containments.

The curves in Theorem~\ref{thm:conj_intro} behave as generally as possible for curves admitting a $\g{r}{d}$. Capturing this idea, we say that a curve $C\in\BN{g}{r}{d}$ is
\emph{$\g{r}{d}$-general} if any additional Brill--Noether special
divisors on $C$ are those determined by Serre duality or by the trivial
containments of Brill--Noether loci, see Definition~\ref{def:grd_gen}. We are naturally led to the following question generalizing the Maximal Brill--Noether Loci conjecture.

\begin{questionintro}
\label{questionintro}
Does $\BN{g}{r}{d}$ contain a $\g{r}{d}$-general curve?
\end{questionintro}

Theorem~\ref{thm:conj_intro} gives a positive answer for all expected maximal Brill--Noether loci, unless $g=7,8,9$. Our more general results, see Theorems~\ref{thm:numbers} and~\ref{thm:main_K3_0}, can be used to address additional cases. Previously, special cases of
Question~\ref{questionintro} have been investigated, see~\cite{Farkas_2001}. For $r=1$, the refined
Brill--Noether theory for curves of fixed gonality answers this
question, see e.g., \cite[Proposition~1.6]{AH}. Classically, the genus-degree formula for plane curves gives many examples where $\BN{g}{2}{d}$ does not contain $\g{2}{d}$-general curves. More generally, if $C$ can be embedded as a
degree $d$ curve in $\mathbb{P}^r$ and admits a $k$-secant $l$-plane,
then $C$ also carries a $\g{r-l-1}{d-k}$, which, if Brill--Noether
special, shows that $C$ is not $\g{r}{d}$-general. Thus Theorem~\ref{thm:conj_intro} shows that curves in maximal Brill--Noether loci do not admit unexpected secants, cf.\
\cite[Remark~6.5]{AH}.

\begin{corintro}
	In every genus $g \geq 3$, with $g \neq 7,8,9$, each expected maximal
	Brill--Noether locus $\BN{g}{r}{d}$ with $r\geq 3$ has a component for
	which a general curve can be embedded as a degree $d$ curve in $\PP^r$
	that does not admit any $k$-secant $l$-plane
	with $k - (k-l-1)(r-l) < 0$.
\end{corintro}
	
The main achievement of this paper is to construct
smooth curves of genus $g$ carrying a special $g^r_d$ and no
$g^{r'}_{d'}$ under certain numerical conditions on the integers $g,r,d,r',d'$.
Our main result is Theorem~\ref{thm:main_K3_0} stating that we can
construct such curves on K3 surfaces with the $g^r_d$ being induced by
a line bundle on the surface. These K3 surfaces live in a divisor
$\KK{g}{r}{d}$ in the moduli space $\K_g$ of quasi-polarized K3 surface of
genus $g$, see \S\ref{subsec:k3notation}. Despite the apparent complexity of the
numerical conditions in Theorem~\ref{thm:main_K3_0}, it provides a very efficient
tool to prove non-containments of Brill--Noether loci in a large
number of new cases, in particular, allowing us to deduce
Theorem~\ref{thm:conj_intro}. This proof, while inspired
by the program initiated in \cite{AH}, does not rely on
any of the special cases of the Maximal Brill--Noether Loci conjecture
treated in the literature so far, and in particular, shows that curves
on K3 surfaces suffice to distinguish the maximal Brill--Noether loci.

We briefly explain the main ideas behind the proof of Theorem~\ref{thm:main_K3_0}. In his famous proof of the Brill--Noether--Petri theorem in \cite{Laz}, Lazarsfeld showed that a Brill--Noether special curve on a K3 surface degenerates to a reducible curve, due to the non-simplicity of the associated Lazarsfeld--Mukai bundle. For many $(g,r,d)$, we prove, in Proposition~\ref{prop:UD1}, that for the very general element $(S,H)\in\KK{r}{g}{d}$, smooth curves $C\in|H|$ degenerate to a reducible curve $C_1 \cup C_2$ in an essentially unique way, a property we call \emph{decomposition rigidity}, cf.\ Definition~\ref{def:DR}. We prove that along this unique degeneration, the limit of any Brill--Noether special divisor on $C$ is highly constrained, allowing us to rule out the existence of additional Brill--Noether special divisors. More precisely, the genera of the components $C_1$ and $C_2$ are $r$ and $g+r-d-1$, respectively, and any Brill--Noether special divisor $\g{r'}{d'}$ on $C$ has ``limits'' $\g{r_1}{d_1}$ and $\g{r_2}{d_2}$ that are Brill--Noether general on the two components, see Theorem~\ref{thm:numbers}. This imposes numerical restrictions on $r_i$ and $d_i$, and thus on $g,r,d,r',d'$. In fact, we show that we can recover a $\g{r'}{d'}$ on $C$ if the numerical constraints are satisfied, so that Theorem~\ref{thm:numbers} gives necessary and sufficient conditions for the existence of a $\g{r'}{d'}$ on a curve in $|C|$. In particular, this can be interpreted as a ``regeneration theorem'' for Brill--Noether special linear systems on curves on K3 surfaces.

	\subsection*{Outline}
	In \S\ref{sec:bn loci}, we recall the definitions
	of expected maximal Brill--Noether loci, the history of progress on distinguishing Brill--Noether
	loci, as well as background on K3 surfaces and Lazarsfeld--Mukai bundles.
	In \S\ref{sec:K3-basic},  we introduce the notion of decomposition rigidity,  and then prove our main technical result, Theorem~\ref{thm:numbers}, in which we study the limiting behavior of Brill--Noether special divisors on curves on K3 surfaces satisfying decomposition rigidity. In \S\ref{sec:K3-ld}, we give sufficient conditions for K3 surfaces in $\K^r_{g,d}$ to satisfy
	decomposition rigidity, which hold for those associated to expected maximal Brill--Noether loci. In \S\ref{sec:numbers}, we study the numerical conditions forced on limits of Brill--Noether special divisors. Finally, in \S\ref{sec:app}, we prove our main
	result, Theorem~\ref{thm:main_K3_0}, and deduce Theorem~\ref{thm:conj_intro}. We conclude \S\ref{sec:app} with
	additional new applications of Theorem~\ref{thm:main_K3_0} to
	non-maximal Brill--Noether loci.

\subsection*{Acknowledgments} 
We would like to thank Steve Fan and Margherita Lelli-Chiesa for
helpful discussions. A.A. received partial support from Simons
Foundation grant 712097 and National Science Foundation grant
DMS-2200845.  R.H. thanks the Hausdorff Research Institute for
Mathematics funded by the Deutsche Forschungsgemeinschaft (DFG, German
Research Foundation) under Germany's Excellence
Strategy--EXC-2047/1--390685813, and the Max--Planck Institut f\"ur
Mathematik Bonn for their hospitality and financial support. The
authors also thank the Department of Mathematics at Dartmouth College,
whose Shapiro Visitor Program funded a long-term visit by A.L.K.\ to
Dartmouth, where this work began, as well as the Department of
Mathematics at Roma Tre and the Trond Mohn Foundation's project ``Pure
Mathematics in Norway'', which funded, respective visits by R.H. and
A.L.K.\ to Rome, where this work was continued.

	\section{Background on Brill--Noether loci and K3 surfaces}
	\label{sec:bn loci}
	
	\subsection{Brill--Noether loci}\label{subsec:BNloci}
	Surprisingly little is known about the geometry of Brill--Noether loci
	in general. The expected codimension of $\BN{g}{r}{d}$ in $\M_g$ is
	$-\rho(g,r,d)$, and it is known that the codimension of any
	component of $\BN{g}{r}{d}$ is at most $-\rho(g,r,d)$. More is known about the existence of components of expected dimension,
	see for example
	\cite{BH_2024_BN_deg,KLM,pflueger_legos,Ser,bigas2023brillnoether}.  However,
	equidimensionality of components is only known when $-3 \le \rho(g,r,d)\le -1$ (additionally assuming
	$g\geq 12$ when $\rho(g,r,d)=-3$)
	\cite{EdidinThesis,steffen_1998}. Complicating the picture, components
	of larger than expected dimension can exist, examples include
	Castelnuovo curves, see for example \cite[Remark 1.4]{pflueger_legos}.
	As mentioned, irreducibility is known for all Brill--Noether loci
	with $\rho=-1$ by \cite{EisenbudHarrisBN-1} as well as all
	$\BN{g}{2}{d}$ with $\rho=-2$ by \cite{CHOI2022}.

	There are various containments known among Brill--Noether loci.  For
	example, Clifford's theorem implies that
	$\BN{g}{r}{2r} \subset \BN{g}{1}{2}$.  There are \emph{trivial
		containments} $\BN{g}{r}{d}\subseteq \BN{g}{r}{d+1}$ obtained by
	adding a basepoint to a $\g{r}{d}$ on $C$; and
	$\BN{g}{r}{d}\subseteq \BN{g}{r-1}{d-1}$ when $r\geq 2$ obtained by
	subtracting a non-basepoint, cf.\
	\cite{Farkas2000,Lelli-Chiesa_the_gieseker_petri_divisor_g_le_13}. Modulo
	these trivial containments, the first two authors defined in \cite{AH}
	the \emph{expected maximal} Brill--Noether loci as follows.
	
	\begin{defn}
		\label{defn:expected_maximal}
		A Brill--Noether locus $\M^r_{g,d}$ is said to be {\em expected
			maximal} if $2 \leq d \leq g-1$, $\rho(g,r,d)<0$, $\rho(g,r,d+1)
		\geq 0$, and $\rho(g,r-1,d-1) \geq 0$ if $r \geq 2$. 
	\end{defn}
	We will say that a triple $(g,r,d)$ is \emph{associated to an expected
		maximal Brill--Noether locus} if $\BN{g}{r}{d}$ is expected maximal.

	We remark that, after accounting for Serre duality
	which gives $\BN{g}{r}{d}=\BN{g}{g-d+r-1}{2g-2-d}$, every
	Brill--Noether locus is contained in at least one
	expected maximal Brill--Noether locus. 
	
	The first two authors then posed a conjecture identifying the maximal Brill--Noether loci.
	
	\begin{conjecture}[{\cite[Conjecture 1]{AH}}]\label{Conj Max BN loci}
		In every genus $g\geq 3$, the maximal Brill--Noether loci are the expected maximal loci, except when $g=7,8,9$.
	\end{conjecture}
	In other words, the expected maximal Brill--Noether loci
	should be maximal with respect to containment, except when $g = 7, 8,
	9$. In particular, being maximal with respect to containment asks that for every pair of expected maximal Brill--Noether loci $\BN{g}{r}{d}$ and $\BN{g}{r'}{d'}$, there exists a curve $C\in\BN{g}{r}{d}$ such that $C\notin \BN{g}{r'}{d'}$. A priori, the curve $C$ could depend on the pair of expected maximal Brill--Noether loci. However, Theorem~\ref{thm:conj_intro} shows the existence of a component whose general member demonstrates all of the non-containments.

We recall the details of the exceptional cases.
	
\begin{example}[Unexpected containments of expected maximal Brill--Noether loci, cf.\ {\cite[Props. 6.2-4]{AH}}] \label{ex: unex cont} 
	We have $\M^2_{7,6} \subset \M^1_{7,4}$, $\M^1_{8,4} \subset\M^2_{8,7}$ and $\M^2_{9,7} \subset\M^1_{9,5}$, and these containments are strict. 
	
	The containment  $\M^2_{7,6} \subset \M^1_{7,4}$ follows from the fact that a $g^2_6$ on a smooth curve of genus $7$ cannot be very ample by the genus formula for plane curves, whence the curve must have a $g^1_4$. The containment $\M^2_{9,7} \subset \M^1_{9,5}$ follows for the same reason. (These were pointed out by H. Larson.)
	
	The containment $\M^1_{8,4} \subset \M^2_{8,7}$ was proven in \cite[Lemma 3.8]{Mu2}. Let $C$ be a smooth curve with a $g^1_4$. Its Serre adjoint is a $g^4_{10}$. If it is not very ample, $C$ will have a $g^3_8$, whose Serre adjoint is a $g^2_6$, whence $C$ carries a $g^2_7$. If the $g^4_{10}$ is very ample, it embeds $C$ into  $\PP^4$ as a curve of degree $10$, where it has $8$ trisecant lines by the Berzolari formula. Hence $C$ has a $g^2_7$.
\end{example}
	
	Prior to the formulation of Conjecture~\ref{Conj Max BN loci} various
	non-containments between Brill--Noether loci were known.  For example, generalizing work of Eisenbud--Harris~\cite{eisenbud_harris_Kodaira} and Farkas~\cite{Farkas2000} in genus~$23$, Choi, Kim, and Kim~\cite{CHOI2012377,CHOI20141458,CHOI2022} showed
	that Brill--Noether loci with $\rho=-1,-2$ in every genus have distinct
	support. Already, this is sufficient to prove Conjecture~\ref{Conj Max
		BN loci} for any genus $g$ such that $g+1$ or $g+2$
	is of the form $\operatorname{lcm}(1,2,\dots,n)$ for some $n\geq 3$,
	the first few being $g=4,5, 10, 11, 58,59, 418, 419, 838, 839, 2518,
	2519, \dotsc$.  Choi and Kim~\cite{CHOI2022} showed
	that Brill--Noether loci with $\rho = -2$  are
	never contained in each other nor in certain other Brill--Noether
	divisors. Lelli-Chiesa~\cite{lelli_rank2} and the first two authors~\cite{AH} also showed
	various non-containments via the lifting of line bundles on curves on K3
	surfaces. As conjectured by Pflueger~\cite{pflueger} and proved by Jensen--Ranganathan~\cite{jensen_ranganathan} and Cook-Powell--Jensen~\cite{cook-powell_jensen}, the combinatorial formula for the expected dimension of Hurwitz--Brill--Noether loci implies that the expected maximal Brill--Noether loci with $r=1$ are not contained in any other expected maximal loci, except when $g=8$, see~\cite[Proposition~1.6]{AH}. This result is now part of the full Brill--Noether theory for curves of fixed gonality, established by Larson, Larson, and Vogt~\cite{larson_refined_BN_Hurwitz,larson_larson_vogt_2020global}. More recently, the gonality
	stratification of $\M_g$ (work of Larson and the first two authors~\cite{AHL}), strata of differentials (work of Bud~\cite{bud2024brillnoether}), and limit
	linear series (work of Teixidor i Bigas~\cite{bigas2023brillnoether} and Bud and the second author~\cite{BH_2024_BN_deg}) were also used to show new non-containments between
	Brill--Noether loci, which all together was sufficient to prove Conjecture~\ref{Conj Max BN
		loci} for $g \leq 23$.

		The idea of Conjecture~\ref{Conj Max BN loci} is that curves in expected maximal Brill--Noether loci should behave as generally as possible, given that they carry a special linear system. More generally, we are interested in Brill--Noether special curves that admit no further Brill--Noether special divisors, of course allowing those forced by trivial containments of Brill--Noether loci and Serre duality.
		
		\begin{defn}\label{def:grd_gen}
			We say a curve $C\in\BN{g}{r}{d}$ is \emph{$\g{r}{d}$-general} if the only Brill--Noether special divisors $\g{r'}{d'}$ with $r'\geq 1$ and $d'\leq g-1$ on $C$ are of the form $(r',d')=(r-i,d-i)$ for $0\leq i \leq r-1$ or $(r',d')=(r,d+j)$ for $0\leq j \leq g-1-d$.
		\end{defn}

		In this language, Theorem~\ref{thm:conj_intro} says that every expected maximal Brill--Noether locus $\BN{g}{r}{d}$ contains a $\g{r}{d}$-general curve, unless $(g,r,d)=(7, 2, 6), (8,1,4), (9,2,7)$. This leads naturally to Question~\ref{questionintro}, describing when other Brill--Noether loci contain a $\g{r}{d}$-general curve. We note that containing a $\g{r}{d}$-general curve is an open condition in $\BN{g}{r}{d}$. 
		Various non-trivial containments of Brill--Noether loci give examples of $\BN{g}{r}{d}$ not containing $\g{r}{d}$-general curves. The genus-degree formula for plane curves gives containments of the form $\BN{g}{2}{d}\subseteq \BN{g}{1}{d-2}$, giving many examples of $\BN{g}{2}{d}$ not containing $\g{2}{d}$-general curves.
		Generalizing this, the existence
		of a $k$-secant $l$-plane to the image of a curve under a $\g{r}{d}$
		gives a $\g{r-l-1}{d-k}$, which, if Brill--Noether special, gives
		further examples of curves that are not $\g{r}{d}$-general. The exceptional
		cases in Example~\ref{ex: unex cont} to Conjecture~\ref{Conj Max BN loci} are exactly of this form.  
		
		 The refined Brill--Noether theory for curves of fixed gonality provides a complete answer to Question~\ref{questionintro} for $r=1$. As proved in~\cite{jensen_ranganathan}, the general smooth projective $k$-gonal curve $C$ of genus $g$ admits a $\g{r}{d}$ if and only if Pflueger's Brill--Noether number \[\rho_k(g,r,d)\coloneqq \rho(g,r,d)+ \underset{0 \leq \ell \leq \min\{r,~g-d+r+1\}}{\max} (g-k-d+2r+1)\ell-\ell^2\] is non-negative. Hence $\BN{g}{1}{k}$ contains a $\g{1}{k}$-general curve if and only if $\rho_k(g,r,d)<0$ for all $(g,r,d)$ associated to expected maximal Brill--Noether loci with $r\geq2$.
		
		\begin{example}\label{ex:fix_gon_grd_gen}
			The expected maximal locus $\BN{g}{1}{k}$ has $k=\floor{\frac{g+1}{2}}$, and contains a $\g{1}{k}$-general curve unless $g=8$, as shown in \cite[Proposition~1.6]{AH}. For  sub-maximal gonality strata, the situation becomes more complicated. For example, if $k=\floor{\frac{g-1}{2}}\geq 2$, then one can verify that $\BN{g}{1}{k}$ contains a $\g{1}{k}$-general curve unless $g=6,7,8,10,11,14$.
		\end{example}
	
	We recall the numerical classification of expected maximal Brill--Noether loci from \cite{AHL}.
	\begin{lemma}[{\cite[Lemma 1.1]{AHL}}] \label{rbound}
		An expected maximal
		 Brill--Noether locus $\BN{g}{r}{d}$ exists for some $d$ if and only if
		\begin{equation} \label{rcases} 1\leq r \leq \begin{cases}
				\floor{\sqrt{g}} & \text{if~} g\geq\floor{\sqrt{g}}^2+\floor{\sqrt{g}}\\
				\floor{\sqrt{g}}-1 & \text{if~} g<\floor{\sqrt{g}}^2+\floor{\sqrt{g}}.
			\end{cases}
		\end{equation}
	\end{lemma}
	
	Once a rank $r$ satisfying the conditions of Lemma~\ref{rbound} is fixed,
	the unique degree $d$ that makes $\M^r_{g,d}$ expected maximal is the largest $d$ such that $\rho(g, r, d) < 0$, namely 
	\begin{equation} \label{dmaxdef}
		d = \dmax(g, r) \coloneqq  r+\left\lceil \frac{gr}{r+1} \right\rceil -1.
	\end{equation}
	
	In other words, Lemma~\ref{rbound} says that the
	expected maximal Brill--Noether loci in $\M_g$ are precisely
	the $\BN{g}{r}{d}$ for $r$ satisfying \eqref{rcases} and $d = \dmax(g,r)$.
	
	We remark for later use that an immediate
	consequence of \eqref{rcases} is the inequality 
\begin{equation} \label{eq:EM3}
g \geq r^2+r
\end{equation}
for $(g,r,d)$ associated to an expected maximal Brill--Noether locus,  which can also directly be deduced from the facts that  $d\leq g-1$ and $\rho(g,r,d+1)\geq 0$.

	\subsection{K3 surfaces} \label{subsec:k3notation}

	We will work with quasi-polarized K3 surfaces of genus $g$, that is, with pairs $(S,H)$ where $S$ is a K3 surface and $H\in\Pic(S)$ is a primitive big and nef line bundle such that $H^2=2g-2$, so that all smooth irreducible curves in $|H|$ have genus $g$. To distinguish Brill--Noether loci, we want such curves to carry a $\g{r}{d}$, which we ensure by specifying the Picard group. In the moduli space $\K_g$ of primitively quasi-polarized K3 surfaces of genus $g$, the Noether--Lefschetz locus consists of K3 surfaces with Picard rank $>1$. Via Hodge theory, the Noether--Lefschetz locus is a union of countably many irreducible Noether--Lefschetz divisors. For $g\geq 2$, $r\geq 0$, and $d\geq 0$, we denote by $\KK{g}{r}{d}$ the Noether--Lefschetz divisor parameterizing quasi-polarized K3 surfaces $(S,H)\in\K_g$ such that the lattice $\Lambda^r_{g,d}=\ZZ[H]\oplus\ZZ[L]$ with intersection matrix 
\begin{equation}
\begin{bmatrix}\label{eq:int_num}
	H^2 & H \cdot L \\
	L \cdot H & L^2
\end{bmatrix}
=
\begin{bmatrix}
	2g-2 & d \\
	d & 2r-2
\end{bmatrix}
\end{equation}
	admits a primitive embedding in $\Pic(S)$ preserving $H$ (using the notation from \cite{AH,Haburcak}). 
	
	It is well known that $\KK{g}{r}{d}$ is non-empty if and only if the discriminant 
	\begin{equation} \label{eq:defD_grd}
		\Delta^r_{g,d}\coloneqq 4(g-1)(r-1)-d^2 <0.
	\end{equation} Indeed, the Hodge index theorem implies $\Delta^r_{g,d}<0$ for any $(S,H)\in\KK{g}{r}{d}$. Surjectivity of the period map (see~\cite[Thm. 2.9(i)]{Mo} or \cite{Ni}) shows there exists a K3 surface $S$ with $\Pic(S)=\Lambda^r_{g,d}$. Acting with \emph{Picard--Lefschetz reflections} on $\Lambda^r_{g,d}$, and using  \cite[VIII, Prop. 3.9]{BPHV}, we may assume that $H$ is nef. In particular, for $(g,r,d)$ satisfying \eqref{eq:defD_grd}, $\KK{g}{r}{d}$ is non-empty and the very general $(S,H)\in\KK{g}{r}{d}$ has $\Pic(S)=\Lambda^r_{g,d}$. We note that when $\KK{g}{r}{d}$ is non-empty, it is an irreducible divisor of $\K_g$, see~\cite{OGrady}. We remark that the $\K^r_{g,d}$ with $\rho(g,r,d)<0$ are precisely the Noether--Lefschetz divisors parameterizing the Brill--Noether special primitively polarized K3 surfaces in the sense of Mukai~\cite{mukai:Fano_threefolds}, see~\cite{GLT,Haburcak}.

Up to Serre duality, we can assume that $(g,r,d)$ associated to a Brill--Noether locus satisfies $d\leq g-1$. Hence, for distinguishing Brill--Noether loci, we will focus on K3 surfaces in $\KK{g}{r}{d}$ satisfying
	\begin{equation} \label{eq:basiccond}
		g \geq 3, \;\; r \geq1, \;\; 2 \leq d \leq g-1.
	\end{equation}
We are interested in $\K^r_{g,d}$ because curves in $|H|$ carry a $\g{r}{d}$. We make this more precise as follows (cf. \cite[Lemma 1.1]{AH}.

\begin{lemma}\label{lemma:restrict_grd}
	Let $(S,H)\in\K^r_{g,d}$ and assume that \eqref{eq:basiccond} holds. Then any irreducible $C \in |H|$ has genus $g$ and carries a $g^r_d$, and $|\O_C(L)|$ is a $\g{r}{d}$ if and only if $h^1(L)=h^1(H-L)=0$.
	
	Furthermore, if $\Pic(S) \cong \Lambda^r_{g,d}$, then 
	$|H|$ contains smooth irreducible curves and $|\O_C(L)|$ is a base point free $\g{r}{d}$ for any smooth irreducible $C \in |H|$.
\end{lemma}

\begin{proof}
	We note that $L^2 \geq 0$, $(H-L)^2=2(g-d-2+r) \geq 2(r-1) \geq 0$,
	$L \cdot H=d >0$ and  $(H-L)\cdot H=2(g-1)-d \geq g-1 \geq 2$, whence by Riemann-Roch and Serre duality, $L$ and $H-L$ are both effective and nontrivial. In particular, $h^0(L-H)=0$ and $\dim |L|=\frac{1}{2}L^2+1 +h^1(L) =r+h^1(L)\geq r$.
	The first assertion in the lemma then follows by adjunction and the sequence
	\[
	\xymatrix{
		0 \ar[r] & \O_S(L-H) \ar[r] & \O_S(L) \ar[r] & \O_C(L) \ar[r] & 0.}
	\] 
	
	By classical results \cite[Cor. 3.2, \S \;2.7]{SD}, a line bundle $M$ on a K3 surface is globally generated if and only if there exists no curve $\Gamma$ such that $\Gamma^2=-2$ and $\Gamma.M <0$ or $\Gamma^2=0$ and $\Gamma \cdot M=1$. Hence, if
	$\Pic(S) \cong \Lambda^r_{g,d}$ one may explicitly check, arguing as in e.g. \cite{Kn},  that $L$ and $H-L$ are globally generated, except in one case: $g=4r-2$, $d=g-1=4r-3$ and $H \sim 2D+\Gamma$, with
	$D \sim L$ or $D \sim H-L$ globally generated and $\Gamma$ a $(-2)$-curve such that $D \cdot \Gamma=1$; in this case, $\Gamma$ is the base divisor of $|H-D|$.
	(We leave the details to the reader, since we will not make use of the last assertion in the lemma in this work.) Since $\Gamma \cdot H=0$, we see that $\O_C(L) \cong \O_C(H-L) \cong \O_C(D)$ for any $C \in |H|$. Thus, in all cases, $\O_C(L)$ is base point free.

	Since $L$ and $H-L$ are part of a basis of $\Pic(S)$, they cannot be a multiple of an elliptic curve. Therefore, again by classical results \cite[Prop. 2.6]{SD}, $L$ and $H-L$  are represented by an irreducible curve whenever they are globally generated, giving the vanishings $h^1(L)=h^1(H-L)=0$. In the one remaining case,
	they have the form $D+\Gamma$ with $D$ represented by an irreducible curve and $\Gamma$ an irreducible curve intersecting $D$ in one point. Hence, $h^1(D+\Gamma)=0$ in this case as well, due to $1$-connectedness. A similar reasoning shows that $H$ is globally generated and represented by an irreducible curve. Hence, by Bertini and the first part of the lemma, there are smooth irreducible curves $C \in |H|$, and $|\O_C(L)|$ is a $\g{r}{d}$.
\end{proof}

In the notation of \cite{ABS,Haburcak}, let $\pi:\P_g \to \K_g$ be the universal smooth hyperplane section, whose fiber above $(S,H)$ is the set of smooth irreducible curves in $|H|$, and let $\phi:\P_g \to \M_g$ be the forgetful map. Also denote by $\pi: \P^r_{g,d}\to \K^r_{g,d}$ the restriction to the divisor $\K^r_{g,d}\subset \K_g$. By Lemma~\ref{lemma:restrict_grd}, the image of $\phi$ restricted to $\P^r_{g,d}$ lies in $\BN{g}{r}{d}$. We summarize this in the following diagram.
\[
\xymatrix@R=7pt@C=16pt{
	& \mathcal{P}^r_{g,d} \ar[rd]^(0.45){\phi} \ar[ld]_(0.45){\pi} & \\
	\mathcal{K}^r_{g,d} & & \mathcal{M}^r_{g,d}
}
\]

Our proof of Theorem~\ref{thm:conj_intro} will show that when restricted to a general element of $\K^r_{g,d}$, the image of $\phi$ consists of $\g{r}{d}$-general curves.
	
	\subsection{Bundles on K3 surfaces}\label{subsec:bun k3}
	Let $S$ be a K3 surface and $C\subset S$ a smooth irreducible curve of genus $g$. A base point free $\g{r}{d}$ on $C$ gives rise to a \emph{Lazarsfeld--Mukai bundle} on $S$. Denoting by $(\L, V)$ the $g^{r}_{d}$, where $\L$ is a line bundle of degree $d$ on $C$ and $V \subset H^0(\L)$ is an $(r+1)$-dimensional subspace, the Lazarsfeld--Mukai bundle $E_{C,(\L,V)}$ is defined as the dual of the kernel of the evaluation map $V\otimes \O_S \twoheadrightarrow \L$. We will use the following well-known facts about $E_{C,(\L,V)}$, cf. \cite{Laz}.
	
	\begin{itemize}
		\item $\rk(E_{C,(\L,V)})=r+1$, $c_1(E_{C,(\L,V)}) = [C]$, $c_2(E_{C,(\L,V)})=d$;
		\item $h^2(E_{C,(\L,V)})=0$;
		\item $E_{C,(\L,V)}$ is globally generated off a finite set;
		\item $\chi(E_{C,(\L,V)}^\vee \otimes E_{C,(\L,V)})=2(1-\rho(g,r,d))$. In particular, if $\rho(g,r,d)<0$, then $E_{C,(\L,V)}$ is non-simple.
	\end{itemize}

	We conclude with a technical lemma which will be useful.
	\begin{lemma} \label{lemma:c1=0}
		Let $\E$ be a torsion free coherent sheaf on a K3
		surface $S$ that is globally generated off a finite set such that $H^2(\E)=0$ and $c_1(\E)=E$, where $E$ is a smooth elliptic curve. Then $\rk \E=1$ and $\E^{\vee\vee} \cong \O_S(E)$.
	\end{lemma}
	
	\begin{proof}
		The cokernel of the canonical inclusion $\E \subset \E^{\vee\vee}$ is
		supported on a finite set, whence $H^2(\E^{\vee\vee})=0$ and $\E^{\vee\vee}$ is globally generated off a finite set. If $\rk \E=1$, there is nothing to show. If $\rk \E \geq 2$, a general subspace $W \subset H^0(\E^{\vee\vee})$ of $\dim W = \rk \E^{\vee\vee}-1$
		gives rise to an evaluation sequence
		\begin{equation} \label{eq:eval1}
			\xymatrix{
				0 \ar[r] & W \* \O_S  \ar[r] & \E^{\vee\vee} \ar[r] & \O_S(E) \* \I_Z \ar[r] & 0,}
		\end{equation}
		where $Z$ is a $0$-dimensional subscheme of $S$ of length $c_2(\E^{\vee\vee})$. Since $\E^{\vee\vee}$ is globally generated off a finite set we must have $h^0(\O_S(E) \* \I_Z) \geq 2$, whence $Z=\emptyset$. As $\O_S(E)$ is globally generated, \eqref{eq:eval1} shows that $\E^{\vee\vee}$ is also. 
		Hence, by a general position argument, for a general subspace $V \subset H^0(\E^{\vee\vee})$ of $\dim V = \rk \E^{\vee\vee}$ the evaluation map
		$V \* \O_S  \to  \E^{\vee\vee}$ drops rank along a smooth member of $|E|$, which we still denote by $E$, and  we have a short exact sequence
		\begin{equation} \label{eq:eval2}
			\xymatrix{
				0 \ar[r] & V \* \O_S  \ar[r] & \E^{\vee\vee} \ar[r] & \A \ar[r] & 0,}
		\end{equation}
		with $\A$ a line bundle on $E$. Dualizing, we obtain
		\begin{equation} \label{eq:eval3}
			\xymatrix{
				0 \ar[r] & \E^{\vee\vee\vee}\cong \E^\vee \ar[r] & V^\vee \* \O_S  \ar[r] & \omega_E \* \A^\vee  \cong \A^\vee\ar[r] & 0.}
		\end{equation}
		The sequences \eqref{eq:eval2} and  \eqref{eq:eval3} show that both $\A$ and 
		$\A^\vee$ are globally generated, whence $\A \cong \O_E$. Since $h^0(\E^\vee)=h^2(\E^{\vee\vee})=0$, we get from \eqref{eq:eval3} that $\rk \E^{\vee\vee}=\dim V \leq h^0(\O_E)=1$,  a contradiction. 
	\end{proof}
	
\section{Brill--Noether special curves on K3 surfaces} 
\label{sec:K3-basic}
		
The aim of this section is to give a necessary and sufficient criterion to determine
whether a curve on a K3 surface with a special linear series induced
by a line bundle on the surface can contain other special linear
series as well. The main result is summarized in Theorem
\ref{thm:numbers} below.
	
A crucial observation of Lazarsfeld \cite[Lemma~1.3]{Laz} is that a
Brill--Noether special curve $C$ on a K3 surface $S$ admits in its
complete linear system $|C|$ a reducible curve. More precisely, as
will also be clear from the proof of Theorem~\ref{thm:numbers},
one has an effective decomposition
	\begin{equation} \label{eq:declaz}
		C\sim  A+B+T,
	\end{equation}
	with $A$ and $B$ globally generated and nontrivial, and $T$ effective (and possibly zero). Control over the (self-)intersection numbers of components of this
	effective decomposition have been crucial technical inputs into a host
	of results concerning lifting linear systems to line bundles on K3
	surfaces and distinguishing Brill--Noether loci, see~\cite{AH,BKM,DM,lelli_rank2}.
	
	For our purposes, the following will be a convenient definition restricting effective decompositions of $H$.
	
	\begin{defn} 
		\label{def:LD} 
		Let $H$ be a divisor on a K3 surface $S$. We say that $H = A+B$ is a
		\emph{flexible decomposition} if $h^0(A) \geq 2$, $h^0(B) \geq 2$,
		$A^2 \geq -2$, $B^2 \geq -2$, and at least one of $A$ and $B$ has
		self-intersection $\geq 0$.
	\end{defn}

		 Note that on K3 surfaces containing no $(-2)$-curves the flexible decompositions are nothing but the effective, nontrivial decompositions.

		\begin{example}\label{ex:lar_decom}
			For K3 surfaces $(S,H) \in \K^r_{g,d}$ satisfying \eqref{eq:basiccond} we have that	$H  \sim  L + (H-L)$ is a  flexible  decomposition. Indeed, since $(H-L)^2= 2(g+r-d-2)\geq 0$ and $(H-L) \cdot H=2g-2-d>0$,  we get by Riemann--Roch that $h^0(H-L)  \geq 2$.  Similarly for $L$.
		\end{example}
		
		We will  henceforth  be interested in the uniqueness of  flexible 
		decompositions.	 The following is  a weakening of the condition that $|H|$ contains no reducible curves  from \cite[Theorem]{Laz}.
		
		\begin{defn} \label{def:DR}
			For a K3 surface $S$ and a divisor $H$, we say that $(S,H)$ satisfies \emph{decomposition rigidity} if $H$ admits  at most one flexible decomposition. 
		\end{defn}
	
	We  will  show in Proposition~\ref{prop:UD1} that  decomposition rigidity  is satisfied for general members $(S,H)\in\K^r_{g,d}$ for many $(g,r,d)$,  in particular for $(g,r,d)$ associated to expected maximal Brill--Noether loci  (cf.\ Lemma~\ref{lemma:basicML2} and Remark~\ref{rem:625}).

	The following lemma, that will be central in the proof of Theorem
	\ref{thm:numbers}, shows the connection between  flexible  decompositions and the decompositions of Lazarsfeld's form \eqref{eq:declaz}.
	
        \begin{lemma}\label{lem:decomp_cases}
	Let  $(S,H)\in\K^r_{g,d}$ satisfying \eqref{eq:basiccond} and  decomposition rigidity.   If we can write $H  \sim  A+B+R$, with
	$A,B,R$  effective and nontrivial, and $A^2 \geq 0$, $B^2 \geq 0$, then 
	either
	\begin{enumerate}[label={\normalfont(\alph*)}]
		\item  $g+4r-2d =4$ and  $r>1$,  or 
		\item   $r=1, \; g=2d, \; A= B = L, \; R^2=-2$, $h^0(R)=1$ and $L+R$ is nef.
	\end{enumerate}
\end{lemma}

\begin{proof}
	
	As $(S,H)$ satisfies  decomposition rigidity, Example~\ref{ex:lar_decom} shows that $H  \sim  L+(H-L)$ is the unique  flexible  decomposition. Also  note that $h^0(A) \geq 2$ and $h^0(B) \geq 2$. 
	
	We  first treat the case $R^2 \geq -2$.
	
	We start by proving that
	\begin{equation}
		\label{eq:intwAB}
		R \cdot A \geq -1 \;\; \mbox{and} \;\; R \cdot B \geq -1.     
	\end{equation}
	Indeed, say $R \cdot A \leq -2$. Then $R^2=-2$ and nefness of $H$ implies that $R \cdot B \geq 4$ and one easily computes that $(A-jR)^2 \geq 0$ for $j=1,2$ and $(B+jR)^2>0$ for $j=1,2,3$. Hence,
	\[ H  \sim  A+(B+R)  \sim  (A-R)+(B+R)  \sim  (A-2R)+(B+3R) \]                
	are all  flexible  decompositions, contradicting   decomposition rigidity.  
	Hence, $R \cdot A \geq -1$, and similarly $R \cdot B \geq -1$, proving \eqref{eq:intwAB} 
	
	By \eqref{eq:intwAB}, the two decompositions
	\[ H  \sim  (A+R)+B  \sim  A+(R+B)  \]
	are  flexible, whence   decomposition rigidity   yields that $A  \sim  B$, and both are linearly equivalent to $L$ or $H-L$, so that $R  \sim  H-2A  \sim  \pm(H-2L)$. If $R^2 \geq 0$ or more generally if $h^0(R) \geq 2$, then also $H  \sim  2A+R$ is a  flexible  decomposition, whence   decomposition rigidity   yields that $R  \sim  A$, so that $H  \sim  3A$,   contradicting that primitivity of $H$.  Hence it remains to treat the case $h^0(R)=1$ and $R^2=-2$. The latter is equivalent to $g+4r-2d =4$. Hence, if $r>1$, we are in case  (a). If $r=1$, then $g=2d$ and we have $H  \sim  2A+R$, with $A  \sim  L$ or $A  \sim  H-L$. In the latter case we would have $2L  \sim  H+R$, whence
	\[ g=2d=2L \cdot H =H^2+R \cdot H=2g-2+R \cdot H,\]
	so that $R \cdot H =2-g<0$ (as $g \geq 3$), contradicting nefness of $H$. Hence, $A  \sim  L$ and $H  \sim  2L+R$, which is case  (b), where we have left to prove that $L+R$ is nef. If it were not, there would exist an irreducible $(-2)$-curve $\Gamma$ such that $\Gamma \cdot (L+R)<0$. Since $H$ is nef, we must have $\Gamma \cdot L>0$ and $\Gamma \cdot R \leq -2$, whence $(R-\Gamma)^2 \geq 0$, yielding the contradiction $h^0(R) \geq h^0(R-\Gamma) \geq 2$.
	
	We next treat the case $R^2 <-2$. Since $R$ is effective, it must contain a $(-2)$-curve $\Gamma$ in its support such that $\Gamma \cdot R<0$, and $R-\Gamma$ is still effective and nontrivial. Since $H$ is nef, we must have $\Gamma \cdot (A+B) \geq 1$, and we can without loss of generality assume that $\Gamma \cdot A \geq 1$. Then $(A+\Gamma)^2 \geq 0$ and we may replace the decomposition $H  \sim  A+B+R$ with the decomposition $(A+\Gamma)+B+(R-\Gamma)$. Repeating the process we will eventually reach a decomposition of the form $H  \sim  A'+B'+R'$, with $A',B',R'$ effective and nontrivial, ${A'}^2 \geq 0$, ${B'}^2 \geq 0$ and ${R'}^2=-2$. This reduces us to the case treated above, which means that we are either in case   (a)  again, or $r=1$, $A'  \sim  B'  \sim  L$ and $h^0(R')=1$. We claim that
	$A  \sim  A'$ and $B  \sim  B'$ (so that we end up in case  (b)  again). Indeed, assume by abuse of notation that at the last but one step of the procedure we have $H  \sim A+B+R$, and there is a  
	$(-2)$-curve  $\Gamma$  such that $\Gamma \cdot R<0$, $R-\Gamma$ is still effective and nontrivial and $\Gamma \cdot A>0$. Then $A' \sim A+\Gamma$, $B' \sim B$ and $R' \sim R-\Gamma$.
	Since $0=L^2={A'}^2=(A+\Gamma)^2=A^2+2(\Gamma \cdot A-1)$, we must have $A^2=0$ and $\Gamma \cdot A=1$. Hence $\Gamma \cdot A'=-1$, so that
	\[ \Gamma \cdot H=\Gamma \cdot(2A'+R')=\Gamma \cdot(2A'+R-\Gamma)=\Gamma \cdot R
	<0,\]
	contradicting nefness of $H$.
\end{proof}	
	
	The main result of this section is the following. 

\begin{thm} \label{thm:numbers}
	Let  $(S,H)\in\K^r_{g,d}$ satisfying \eqref{eq:basiccond} and  decomposition rigidity.  Assume further that if $r>1$, then $g+4r-2d \neq 4$.
	Then $|H|$ contains smooth irreducible curves of genus $g$, and for any smooth irreducible $C \in |H|$, the linear system $|\O_C(L)|$ is a base point free complete $g^r_d$, which is very ample if $r \geq 3$. Moreover, there exists a smooth irreducible curve in $|H|$ carrying a $g^{r'}_{d'}$ with $\rho(g,r',d') <0$ if and only if there exist non-negative integers $r_1,r_2,d_1,d_2$ such that
	\begin{eqnarray}
		\label{eq:condi1} r_1+r_2 & = & r'-1 \\
		\label{eq:condi2} d_1+d_2 & \leq & d'-d+2r-2 \\
		\label{eq:condi3} 0 & \leq  & \rho(r,r_1,d_1) <r \\
		\label{eq:condi4} 0 & \leq  & \rho(g+r-d-1,r_2,d_2) <g+r-d-1. 
	\end{eqnarray}
\end{thm}

\begin{proof}
	
As $(S,H)$ satisfies  decomposition rigidity, Example~\ref{ex:lar_decom} shows that $H  \sim  L+(H-L)$ is the unique  flexible  decomposition. Lemma~\ref{lem:decomp_cases}  and our assumptions imply  that
	$L$ and $H-L$ admit no  nontrivial  effective decomposition with at least
	one summand of non-negative self-intersection except in the case $r=1$, $g=2d$, $H -L  \sim  L+R$, with $R^2=-2$, $h^0(R)=1$ and $L+R$ nef. Therefore, well-known results on linear systems on K3 surfaces \cite[Corollary~3.2,~\S2.7]{SD} imply that $L$ and $H-L$ are globally generated and the general members of $|L|$ and $|H-L|$ are smooth irreducible curves of genus $g(L)=r \geq 1$ and $g(H-L)=g+r-d-1 \geq 1$, respectively, and that $L$ is even very ample if $r \geq 3$. Similarly, the general member of $|H|$ is a smooth, irreducible curve of genus $g$. In particular, $h^1(L)=h^1(H-L)=0$, whence $|\O_C(L)|$ is a base point free complete $g^r_d$ for any smooth $C \in |H|$ by Lemma~\ref{lemma:restrict_grd}, and it is even is very ample if $r \geq 3$ (as $L$ is).
	
	We have left to prove the last statement of the proposition. We first prove the ``only if''-part.
	
	Assume that  there exists a smooth curve $C \in |H|$ carrying a $g^{r'}_{d'}$ with $\rho(g,r',d') <0$. Its base point free part is a $g^{r'}_{d''}$ with $d'' \leq d'$. 
	Let $\E$ be the Lazarsfeld--Mukai bundle associated to the base point free part.
	The fact that $\rho(g,r',d'') \leq \rho(g,r',d') <0$ implies that $\E$ is non-simple, whence we can by standard arguments find an endomorphism $\varphi: \E \to \E$ dropping rank everywhere (see~\cite[p.~302]{Laz}). Set $\F\coloneqq \im \varphi$ and $\G\coloneqq \coker \varphi$, which both have positive ranks. Being both a quotient sheaf and subsheaf of $\E$, we have that $\F$ is torsion free, globally generated off a finite set and with $h^2(\F)=0$. In particular, $\F$ is nontrivial, whence $c_1(\F)$ is represented by an effective nonzero divisor $F$ (see, e.g., \cite[Fact at the bottom of p.~302]{Laz}). Since $F$ must be globally generated off a finite set, we must have $F^2 \geq 0$. Let $\T$ be the (possibly zero) torsion subsheaf of $\G$ and $\overline{\G}\coloneqq \G/\T$. Then $\overline{\G}$ is torsion free, and, being a quotient sheaf of $\E$, it is globally generated off a finite set and satisfies $h^2(\overline{\G})=0$. As above, this implies that $c_1(\overline{\G})$ is represented by an effective nonzero divisor $G$, with $G^2 \geq 0$. We thus have
	\[ H  \sim  c_1(\E)  \sim  c_1(\F)+c_1(\G)  \sim  c_1(\F)+c_1(\overline{\G}) + c_1(\T)  \sim F+G+T,\]
	with $T$ an effective (possibly zero) divisor supported on the support of $\T$.

	Letting $\overline{\F}$ denote the kernel of the surjection $\E \to \G \to \overline{\G}$, we have a commutative diagram
	
	\begin{equation}\label{eq:diagram}
		\begin{gathered}
			\xymatrix@R=12pt{  &          & 0 \ar[d]      & 0 \ar[d]                  &           &  \\  & 0 \ar[r] & \F\ar[d]\ar[r]& \overline{\F}\ar[d]\ar[r] & \T \ar[r] & 0\\  &          & \E \ar[d]\ar@{=}[r] & \E \ar[d]           &           &  \\0 \ar[r] & \T \ar[r] & \G \ar[d]\ar[r]& \overline{\G} \ar[d] \ar[r]&0&\\  &          & 0            & 0                          &           &  \\}
		\end{gathered}		
	\end{equation}
	Note  that $\overline{\F}$ is torsion free, being a subsheaf of $\E$. The rightmost vertical sequence in the diagram, together with the facts that $\E$ is locally free and $\overline{\G}$ is torsion free, yields that $\overline{\F}$ is \emph{normal} (cf., e.g., \cite[II, Lemma 1.1.16]{OSS}), whence \emph{reflexive} (cf., e.g., \cite[II, Lemma 1.1.12]{OSS}), whence locally free. 
	
	Also note that $h^2(\overline{\F})=0$, as $h^2(\F)=0$.
	
	Consider the standard exact sequence
	\begin{equation} \label{eq:dd}
		\xymatrix{ 0 \ar[r] & \overline{\G} \ar[r]  & \overline{\G}^{\vee\vee} \ar[r] & \tau_{\overline{\G}} \ar[r] & 0,}
	\end{equation}
	where $\overline{\G}^{\vee\vee}$ is locally free and $\tau_{\overline{\G}}$ has finite support. In particular, also
	$\overline{\G}^{\vee\vee}$ is globally generated off a finite set and $h^2(\overline{\G}^{\vee\vee})=0$. Combining the  rightmost vertical  sequence in \eqref{eq:diagram} with \eqref{eq:dd}, we obtain
	\[
	\xymatrix{
		0 \ar[r] & \overline{\F} \ar[r]  & \E \ar[r] & \overline{\G}^{\vee\vee} \ar[r]
		& \tau_{\overline{\G}} \ar[r] & 0
	}
	\]
	where $\overline{\F}$ and $\overline{\G}^{\vee\vee}$ are both locally free, $\overline{\G}^{\vee\vee}$ is globally generated off a finite set, $h^2(\overline{\F})=h^2(\overline{\G}^{\vee\vee})=0$, and  $\tau_{\overline{\G}}$ supported on a finite set.
	\begin{claim}
		Also $\overline{\F}$ is globally generated off a finite set. Moreover, $c_1(\overline{\F})$ and $c_1(\overline{\G}^{\vee\vee})$ are linearly equivalent to $L$ and $H-L$.
	\end{claim}
	\begin{proof}[Proof of claim]
		If $T=0$, then $\T$ is supported on a finite set. Since $\F$ is also globally generated off finitely many points, the  upper horisontal  vertical sequence in \eqref{eq:diagram} shows that $\overline{\F}$ is again globally generated off a finite set. Moreover, $c_1(\overline{\F})$ and $c_1(\overline{\G}^{\vee\vee})$ are linearly equivalent to $L$ and $H-L$ by  decomposition rigidity.
		
		If $T \neq 0$, then the hypotheses and Lemma~\ref{lem:decomp_cases} yield that $r=1$, $g=2d$, $F \sim  G  \sim  L$ and $T \sim R \sim  H-2L$, with $R^2=-2$.  In particular, $L$ is represented by a smooth elliptic curve.  By Lemma~\ref{lemma:c1=0} we obtain that
		$\rk \overline{\G}^{\vee\vee}=\rk \G=1$ and $\rk \overline{\F} =\rk \F=1$, whence $r'=\rk \E-1=1$. Hence $\overline{\G}^{\vee\vee} \cong L$ and $\overline{\F} \cong H-L$. In particular, $\overline{\F}$ is globally generated and the claim follows.\end{proof}
	
	To simplify notation, we can therefore assume that we have an exact sequence
	\begin{equation} \label{eq:finalass}
		\xymatrix{
			0 \ar[r] & \F \ar[r]  & \E \ar[r] & \G \ar[r]
			& \tau\ar[r] & 0
		}
	\end{equation}
	with $\F$ and $\G$ both locally free of positive ranks and globally generated off finite sets, such that $h^2(\F)=h^2(\G)=0$ and  $c_1(\F)$ and $c_1(\G)$ are represented by $L$ and $H-L$, and  $\tau$ is supported on a finite set. 
	
	We set $r_{\F}\coloneqq\rk\F-1$ and $r_{\G}\coloneqq\rk \G-1$, $d_{\F}\coloneqq c_2(\F)$, $d_{\G}\coloneqq c_2(\G)$, $g_{\F}\coloneqq \frac{1}{2}c_1(\F)^2+1$ and $g_{\G}\coloneqq \frac{1}{2}c_1(\G)^2+1$. Then
	$r_{\F} \geq 0$ and $r_{\G} \geq 0$. Since $\F$ and $\G$ are globally generated off finite sets, we also have $d_{\F}\geq 0$ and $d_{\G} \geq 0$. Since $\{c_1(\F),c_1(\G)\}=\{L,H-L\}$, we have
	\begin{equation} \label{eq:ggg}
		\{g_{\F},g_{\G}\}=\{g(L),g(H-L)\}=\{r,g+r-d-1\}.
	\end{equation}
	
	From  \eqref{eq:finalass} we find
	\[ r'+1=\rk \E  =\rk\F + \rk\G=r_{\F}+1+r_{\G}+1,\]
	that is,
	\begin{equation}
		\label{eq:cond0}
		r_{\F}+r_{\G}=r'-1;
	\end{equation}
	moreover,
	\begin{eqnarray*}
		d' & \geq  &  d''=  c_2(\E)=c_2(\F)+c_2(\G)+c_1(\F) \cdot c_1(\G)+\length \tau \geq 
		d_{\F}+d_{\G}+L \cdot(H-L) \\
		& = &  d_{\F}+d_{\G}+d-2r+2,
	\end{eqnarray*}
	in other words
	\begin{equation}
		\label{eq:cond1}
		d_{\F}+d_{\G} \leq d'-d+2r-2.
	\end{equation}
	
	If $r_{\F}=0$, then $\rk \F=1$, so that $d_{\F}=c_2(\F)=0$. This implies that $\rho(g_{\F},r_{\F},d_{\F})=\rho(g_{\F},0,0)=0$. If $r_{\F} >0$, then, as $\F$ is globally generated off a finite set, we have a short exact sequence
	\begin{equation} \label{eq:standa}
		\xymatrix{
			0 \ar[r] & \CC^{r_{\F}+1} \* \O_S  \ar[r]  & \F \ar[r] & \A_{\F} \ar[r]
			& 0,
		}
	\end{equation}
	with $\A_{\F}$ a torsion free rank-one sheaf on a reduced and irreducible member $D \in |c_1(\F)|$. As $h^2(\F)=0$,  the sheaf  $\F$ is nontrivial, whence $h^0(\F) > \rk(\F)=r_{\F}+1$, so that
	\begin{equation}
		\label{eq:h0A}
		h^0(\A_{\F})>0. 
	\end{equation}
	Dualizing \eqref{eq:standa}, we obtain
	\begin{equation} \label{eq:standa2}
		\xymatrix{
			0 \ar[r] & \F^\vee \ar[r] & \CC^{r_{\F}+1} \* \O_S  \ar[r]  & \B_{\F}\coloneqq \ext^1(\A_{\F},\O_S) \ar[r]
			& 0,
		}
	\end{equation}
	again with $\B_{\F}$ a torsion free rank-one sheaf on $D$, which is globally generated by construction and satisfies $\deg(\B_{\F})=c_2(\F)=d_{\F}$. Since $h^0(\F^\vee)=h^2(\F)=0$, we see that
	$h^0(\B_{\F}) = r_{\F}+1+h^1(\F)$. Now choose $h^1(\F)$ general points $x_1,\ldots,x_{h^1(\F)}$ in the smooth locus of $D$ and set $\B'_{\F}\coloneqq \B_{\F}(-x_1-\cdots-x_{h^1(\F)})$. Then
	\begin{eqnarray*}
		h^0(\B'_{\F}) & = & h^0(\B_{\F})-h^1(\F)=r_{\F}+1, \\
		h^1(\B'_{\F}) & = & h^1(\B_{\F})=h^0(\ext^1(\B_{\F},\O_S))=h^0(\A_{\F})>0,
	\end{eqnarray*}
	using \eqref{eq:h0A} (cf. \cite[Lemmas 2.1 and 2.3]{Go}). Moreover, $\B'_{\F}$ is still globally generated, so we have a short exact sequence
	\begin{equation} \label{eq:standa3}
		\xymatrix{
			0 \ar[r] & {\F'}^\vee \ar[r] & \CC^{r_{\F}+1} \* \O_S  \ar[r]  & \B'_{\F}
			\ar[r]
			& 0,
		}
	\end{equation}
	defining a new Lazarsfeld--Mukai bundle $\F'$.
	We have $\rk(\F')=\rk(\F)=r_{\F}+1$, $c_1(\F')=c_1(\F)$ and
	$d'_{\F}\coloneqq c_2(\F') =\deg(\B'_{\F})=d_{\F}-h^1(\F) \leq d_{\F}$.
	Dualizing \eqref{eq:standa3} we obtain
	\begin{equation} \label{eq:standa4}
		\xymatrix{
			0 \ar[r] & \CC^{r_{\F}+1} \* \O_S \ar[r] & {\F'}   \ar[r]  &
			\ext^1(\B'_{\F},\O_S)
			\ar[r]
			& 0,
		}
	\end{equation}
	and since $h^0(\ext^1(\B'_{\F},\O_S))=h^1(\B'_{\F})>0$ (cf. \cite[Lemma 2.3]{Go}), we see that $\F'$ is globally generated off a finite set. 
	Recall that $D  \sim L$ or $H-L$, 
	and, as a consequence of what we said in the first lines of the proof, $L$ and $H-L$ admit no decompositions in moving classes. Hence 
	
	\[ \rho(g_{\F},r_{\F}, d'_{\F}) \geq 0,\]
	as in the proof of \cite[Lemma 1.3]{Laz}.
	At the same time, since $h^0(\B'_{\F})>0$ and $h^1(\B'_{\F})>0$, we also have
	\[ \rho(g_{\F},r_{\F}, d'_{\F})=g(D)-h^0(\B'_{\F})h^1(\B'_{\F}) < g(D)=g(\F).\]
	
	To summarize, we have in any event found an integer $d'_{\F}$ such that
	$0 \leq   d'_{\F} \leq d_{\F}$ and
	\begin{equation}
		\label{eq:rho0gf}
		0 \leq \rho(g_{\F},r_{\F}, d'_{\F}) < g_{\F}. 
	\end{equation}
	
	Similarly, considering $\G$, we find an integer $d'_{\G}$ such that
	$0 \leq   d'_{\G'} \leq d_{\G}$ and
	\begin{equation}
		\label{eq:rho0gg}
		0 \leq \rho(g_{\G},r_{\G}, d'_{\G}) < g_{\G}. 
	\end{equation}
	By \eqref{eq:cond1} we see that
	\begin{equation}
		\label{eq:cond1+}
		d'_{\F}+d'_{\G} \leq d'-d+2r-2.
	\end{equation}
	Recalling \eqref{eq:ggg}, we see that \eqref{eq:cond0}, \eqref{eq:cond1+}  and \eqref{eq:rho0gf}-\eqref{eq:rho0gg} imply \eqref{eq:condi1}, \eqref{eq:condi2}  and \eqref{eq:condi3}-\eqref{eq:condi4}, with $\{r_1,r_2\}=\{r_{\F},r_{\G}\}$ and $\{d_1,d_2\}=\{d'_{\F},d'_{\G}\}$.
	
	We then prove the ``if''-part.
	
	Assume the existence of the integers as stated in the  theorem.  Assume first that $r_1,r_2>0$. Then, by \eqref{eq:condi3}-\eqref{eq:condi4}, on any smooth curve in $|L|$ and $|H-L|$ there exist a complete special $g^{r_1}_{d_1}$ and $g^{r_2}_{d_2}$, respectively. Denote by $\A_i$, $i=1,2$, the corresponding line bundles, and by $D_i$ the corresponding curves. Define $\A'_i$ to be the base point free part of $\A_i$, that is, the image of the evaluation map $H^0(\A_i) \* \O_{D_i} \to \A_i$, by $\B_i$ the base point free part of $\omega_{D_i} \* {\A'}_i^{-1}$, and set $\overline{\A}_i\coloneqq \omega_{D_i} \* \B_i^{-1}$. Then $h^0(\overline{\A}_i)=h^0(\A_i)+l_i$ and $\deg(\overline{\A}_i)=\deg(\A'_i)+l_i \leq d_i+l_i$, where $l_i$ is the length of the base locus of $|\omega_{D_i} \* {\A'}_i^{-1}|$. One easily checks, as e.g. in \cite[Lemma 3.1]{Go}, that both $|\overline{\A}_i|$ and  $|\omega_{D_i} \* \overline{\A}_i^{-1}|$ are base point free. Hence, we have found, on (all) smooth curves in $|L|$ and $|H-L|$, a base point free complete $g^{r_i+l_i}_{d'_i+l_i}$, with $d'_i \leq d_i$ and $l_i \geq 0$, for $i=1,2$, respectively, and such that its adjoint is also base point free. 
	
	Let $\F$ and $\G$ be the associated Lazarsfeld--Mukai bundles, of ranks $r_1+l_1+1$ and $r_2+l_2+1$, respectively, which are globally generated by the assumption on the base point freeness of the adjoint linear systems. Set $\E\coloneqq \F \+ \G$. Then
	\[ \rk(\E)=r_1+l_1+r_2+l_2+2=r'+1+l_1+l_2\] by \eqref{eq:condi1} and
	\begin{eqnarray*}
		c_2(\E) & = & c_2(\F)+c_2(\G)+c_1(\F)\cdot c_1(\G) = d'_1+l_1+d'_2+l_2+L \cdot (H-L) \\
		& = & 
		d'_1+l_1+d'_2+l_2+d-2r+2 
		\leq   d_1+l_1+d_2+l_2+d-2r+2 \\
		& \leq  & d'+l_1+l_2,
	\end{eqnarray*}
	by \eqref{eq:condi2}.
	Then, as is well-known, the evaluation map $\CC^{r'+1} \* \O_S \to \E$ drops rank along a smooth curve $C \in |H|$ and the cokernel is a line bundle $\A$ on $C$, such that $|\omega_C-\A|$ is a $g^{r'+l_1+l_2}_{c_2(\E)}$; in particular $C$ carries a $g^{r'+l_1+l_2}_{d'+l_1+l_2}$, whence also a $g^{r'}_{d'}$.
	
	Assume now that $r_1=0$ and $r_2>0$. Then by \eqref{eq:condi4} there exist a complete special $g^{r_2}_{d_2}$ on any smooth curve in $|H-L|$, and as before, we find a complete, base point free $g^{r_2+l_2}_{d'_2+l_2}$ with $d'_2 \leq d_2$ and $l_2 \geq 0$, such that its adjoint is again globally generated. Letting $\F$ be its associated Lazarsfeld--Mukai bundle of rank $r_2+l_2+1$, which again is globally generated, we set $\E\coloneqq L \+ \F$, and argue as above. Similarly if $r_2=0$ and $r_1>0$.
	
	Finally, if $r_1=r_2=0$, we set $\E=L \+ (H-L)$ and argue as before. 
\end{proof}

\begin{remark}\label{rem:tbr}
	A careful look at the proofs of Lemma~\ref{lem:decomp_cases} and Theorem~\ref{thm:numbers} shows that the condition ``$(S,H)\in\K^r_{g,d}$'' can be replaced by ``$(S,H)$ a quasi-polarized K3 surface of genus $g$ with a line bundle $L \in \Pic (S)$ satisfying $L^2=2r-2$, $L.H=d$ and $H\not \sim 3L$''. Also note that whenever $H$ is $n$-divisible in $\Pic(S)$ for an $n \geq 4$, then $(S,H)$ does not satisfy  decomposition rigidity. 
\end{remark}

	\section{ Flexible  decompositions on K3s of Picard number two} \label{sec:K3-ld}

We want to find K3 surfaces satisfying  the  conditions in Theorem~\ref{thm:numbers},  in particular decomposition rigidity.  Throughout  the section we will fix integers $(g,r,d)$ satisfying  \eqref{eq:defD_grd} and \eqref{eq:basiccond}, and a pair $(S,H)\in\K^r_{g,d}$ with $\Pic(S)=\Lambda^r_{g,d}$.  For simplicity we  write $\Delta\coloneqq \Delta^r_{g,d}$.

Assume now that $H \sim  A+B$ is a  flexible  decomposition of $H$. We may write
\[A=xH-yL \;\; \mbox{and}\;\;  B=(1-x)H+yL\;\; \mbox{for some} \;\; x,y\in\ZZ.\]
We may and will assume that $A^2 \geq 0$ and $B^2 \geq -2$. The next two lemmas give necessary conditions on $x$ and $y$.

\begin{lemma} \label{lem:firstldec}
	We have
	\begin{eqnarray}
		\label{eq:A2_2}	 & (g-1)x^2-dxy +(r-1)y^2\geq  0, & \\
		\label{eq:B2_2}  &  (g-1)(1-x)^2+d(1-x)y+(r-1)y^2\geq  -1, & \\
		\label{eq:ABH} & \frac{2(g-1)}{d}(x-1) < y < \frac{2(g-1)}{d}x. &
	\end{eqnarray} 
\end{lemma}
\begin{proof}
	The two first conditions are equivalent to $\frac{1}{2}A^2 \geq 0$ and $\frac{1}{2}B^2\geq -1$, respectively. The last condition is equivalent to $A \cdot H >0$ and $B \cdot H > 0$, which are satisfied  since $h^0(A) \geq 2$ and $h^0(B) \geq 2$ and $H$ is  big and nef. 
\end{proof}	

\begin{lemma}\label{lem:ldecPB}
	If $r=1$, then 
	\[ x =0 \;\; \mbox{and} \;\; -\frac{g}{d} \leq y <0, \;\;\;\;\; \mbox{or} \;\; x=1 \;\; \mbox{and } \;\;  0< y \leq \frac{g-1}{d}. \] 
	
	If $r\geq 2$, then either
	\begin{eqnarray} \label{eq:1K}
		&0< x \leq \max \left\{1+\frac{d}{\sqrt{\abs{\Delta}}\sqrt{g-1}} ~,~  \frac{2(r-1)g}{(d-\sqrt{\abs{\Delta}})\sqrt{\abs{\Delta}}}\right\}  \; \; \mbox{and} &   \\
		\nonumber  & \frac{2(g-1)}{d}(x-1) < y \leq \min\left\{\frac{d-\sqrt{\abs{\Delta}}}{2(r-1)}x\;,\;\frac{g}{\sqrt{\abs{\Delta}}}\right\} &\end{eqnarray} 
	or
	\begin{eqnarray} \label{eq:3K}
		&-\frac{2(r-1)g}{(d+\sqrt{\abs{\Delta}})\sqrt{\abs{\Delta}}} \leq x \leq
		0  \; \; \mbox{and} &  \\
		\nonumber  \;\; 
		&\max\left\{-\frac{g}{\sqrt{\abs{\Delta}}}\;,\; \frac{2(g-1)}{d}(x-1)\right\} \leq y \leq 
		\frac{d+\sqrt{\abs{\Delta}}}{2(r-1)}x, \;\; y \neq 0, \;\; y \neq \frac{2(g-1)}{d}(x-1).&
	\end{eqnarray}
\end{lemma}

\begin{proof}		
	We first treat the case $r=1$.
	
	If $x <0$, then \eqref{eq:A2_2} and the right inequality in \eqref{eq:ABH} yield
	\[ (g-1)x \leq dy < 2(g-1)x,\]
	whence the contradiction $(g-1)x>0$. 
	If $x>0$, then \eqref{eq:A2_2} and  the left inequality in \eqref{eq:ABH} yield
	\[
	2(g-1)(x-1) < dy \leq (g-1)x,
	\]
	which implies $(g-1)x < 2(g-1)$, whence $x=1$, and thus $0< y \leq \frac{g-1}{d}$. 
	
	If $x =0$, then \eqref{eq:B2_2} and  the left inequality in \eqref{eq:ABH} yield $-\frac{g}{d} \leq y <0$. 
	
	The rest of the proof will deal with the case $r\geq 2$.
	
	We first note that \eqref{eq:ABH} implies that
	\begin{equation}
		\label{eq:signs}
		\mbox{either $x \leq 0$ and $y<0$, or $x>0$ and $y>0$.}
	\end{equation}
	Letting $a_{\pm}\coloneqq \frac{d\pm \sqrt{\abs{\Delta}}}{2(r-1)}$, we note that \eqref{eq:A2_2} 
	factors as
	\begin{equation} \label{eq:A2'_2}
		(r-1)(y-a_+ x)(y- a_- x) \geq 0.
	\end{equation}
	Define the following lines in the plane:
	\[ \ell_+: \; y=a_+x, \;\;\;\;  \ell_-: \; y=a_-x, \;\;\;\;  
	\ell: \; y=\frac{2(g-1)}{d}(x-1).\]
	Since $a_+>\frac{2(g-1)}{d}>a_->0$, the conditions \eqref{eq:A2'_2} and \eqref{eq:ABH}  yield, also recalling \eqref{eq:signs}, that either
	\begin{equation}
		\label{eq:1kvad}
		x>0 \;\; \mbox{and} \;\;         0 \leq \frac{2(g-1)}{d}(x-1) < y \leq  a_-x
	\end{equation}
	(which determines a triangle bounded by the $x$-axis and the lines $\ell$ and $\ell_-$ in the first quadrant)
	or
	\begin{equation}
		\label{eq:3kvad}
		x \leq 0 \;\; \mbox{and} \;\;         \frac{2(g-1)}{d}(x-1) < y  \leq  a_+x, \;\; y <0
	\end{equation}                     
	(which determines a triangle bounded by the $y$-axis and the lines $\ell$ and $\ell_+$ in the third quadrant). Furthermore, condition \eqref{eq:B2_2} implies that the points $(x,y)$ must lie in the region containing the origin determined by the hyperbola with two branches\[
	\mathfrak{c}: \; (g-1)(1-x)^2+d(1-x)y+(r-1)y^2=-1,\]
	as shown in the following picture,  which shows the case $(g,r,d)=(14,3,13)$.
	\begin{figure*}[h]
		\centering
		\includegraphics[scale=0.8]{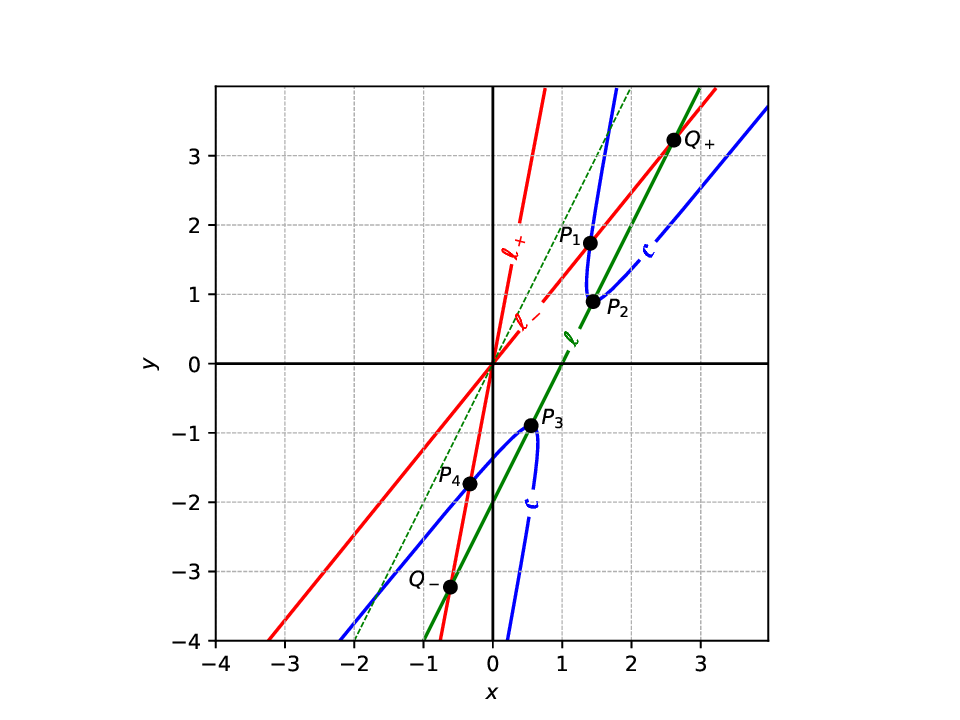}
	\end{figure*}

	We compute the intersection points between the curves, as shown in the picture:
	\begin{eqnarray*}
		Q_+=(x_+,y_+)\coloneqq \ell \cap \ell_- & = & \left(1+\frac{d}{\sqrt{\abs{\Delta}}}, \frac{2(g-1)}{\sqrt{\abs{\Delta}}}\right), \\
		Q_-=(x_-,y_-)\coloneqq \ell \cap \ell_- & = & \left(1-\frac{d}{\sqrt{\abs{\Delta}}}, -\frac{2(g-1)}{\sqrt{\abs{\Delta}}}\right), \\
		P_1=(x_1,y_1)\coloneqq\mathfrak{c} \cap \ell_- &= & \left( \frac{2(r-1)g}{(d-\sqrt{\abs{\Delta}})\sqrt{\abs{\Delta}}} ~,~ \frac{g}{\sqrt{\abs{\Delta}}} \right),\\
		P_4=(x_4,y_4)\coloneqq\mathfrak{c} \cap \ell_+ &=& \left( -\frac{2(r-1)g}{(d+\sqrt{\abs{\Delta}})\sqrt{\abs{\Delta}}} ~,~ -\frac{g}{\sqrt{\abs{\Delta}}} \right), \\
		P_2= (x_2,y_2)\coloneqq\mathfrak{c} \cap \ell \cap\{y\geq 0\} &= & \left( 1+\frac{d}{\sqrt{g-1}\sqrt{\abs{\Delta}}} ~,~ \frac{2\sqrt{g-1}}{\sqrt{\abs{\Delta}}} \right),\\
		P_3=(x_3,y_3)\coloneqq\mathfrak{c} \cap \ell \cap\{y\leq 0\} &=&\left(1-\frac{d}{\sqrt{g-1}\sqrt{\abs{\Delta}}} ~,~ -\frac{2\sqrt{g-1}}{\sqrt{\abs{\Delta}}} \right).
	\end{eqnarray*}
	
	We first consider case \eqref{eq:1kvad}. We note that $x_2 \leq x_+$ and $y_+ \geq y_1 \geq y_2$. Thus, the points $(x,y)$ lie in the region bounded by the $x$-axis, the lines $\ell_-$ and $\ell$, and the segment of $\mathfrak{c}$ between $P_1$ and $P_2$. Since $\mathfrak{c}$ is convex, the points $(x,y)$ are  contained in the pentagon bounded by the $x$-axis, the lines 
	$\ell_-$ and $\ell$, and the line segment $P_1P_2$. In particular, $x \leq \max\{x_1,x_2\}$ and $y \leq y_1$. Together with \eqref{eq:1kvad}, this yields \eqref{eq:1K}. 
	
	We next consider case \eqref{eq:3kvad}. We note that $x_- \leq x_3$ and $y_- \leq y_4 \leq y_3$. The latter implies that the intersection points $P_4$ and $P_3$ occur in the region where $\ell_+$ lies to the left of $\ell$, whence it follows that also $x_4 \leq x_3$.		(Note, however, that contrary to the impression one may get by the picture, we could have $x_3<0$.) The points $(x,y)$ lie in the region bounded by the $y$-axis, the lines $\ell_+$ and $\ell$ and the the segment of $\mathfrak{c}$ between $P_4$ and $P_3$. Since $\mathfrak{c}$ is convex, the points $(x,y)$ are  contained in the pentagon bounded by the $y$-axis, the lines 
	$\ell_+$ and $\ell$, and the line segment $P_4P_3$. In particular, $x \geq x_4$  and $y \leq y_4$. Together with \eqref{eq:3kvad}, this yields \eqref{eq:3K}. 
\end{proof}

\begin{remark} \label{obs:dircomp}
	A direct computation shows that
	\[ \frac{2(r-1)g}{(d-\sqrt{\abs{\Delta}})\sqrt{\abs{\Delta}}}-\frac{2(r-1)g}{(d+\sqrt{\abs{\Delta}})\sqrt{\abs{\Delta}}}=1+\frac{1}{g-1}.\]
	In particular, setting  $m\coloneqq\left \lfloor \max\left\{1+\frac{d}{\sqrt{\abs{\Delta}}\sqrt{g-1}} ~,~  \frac{g(2r-2)}{(d-\sqrt{\abs{\Delta}})\sqrt{\abs{\Delta}}}\right\}\right\rfloor$, Lemma~\ref{lem:ldecPB} shows that \linebreak $x \in \{-m+1,\ldots,m\}$. 
\end{remark}

As a consequence, we can deduce  decomposition rigidity  under certain assumptions.        

\begin{prop} \label{prop:UD1}
	Assume that \eqref{eq:basiccond} holds,
	\begin{equation}
		\label{eq:cliff}  d  \geq 
		\begin{cases}
			\frac{g-3}{2}+2r, &  \mbox{if} \; \;  r\geq 2, \\
			\frac{g}{2}, & \mbox{if} \; \; r=1,                                                           \end{cases}
	\end{equation}
	and 
	\begin{equation}
		\label{eq:caz} (r-1)(3g-4)^2  <  2d^2(g-2). 
	\end{equation}
	Then any $(S,H)\in\K^r_{g,d}$  with $\Pic(S)=\Lambda^r_{g,d}$ satisfies decomposition rigidity.
\end{prop}

\begin{remark} \label{rem:g>6}
	Condition \eqref{eq:cliff} together with the assumption that $d \leq g-1$ (cf. \eqref{eq:basiccond}) imply that $g+1 \geq 4r$.  
\end{remark}

\begin{remark}
	
	If \eqref{eq:cliff} is not fulfilled, then $(H-2L)^2 \geq 
	\begin{cases}
		-2, & \mbox{if} \; \;  r\geq 2, \\
		0, & \mbox{if} \; \; r=1,
	\end{cases}$
	whence either \linebreak $H  \sim  2L+(H-2L)$ is another  flexible  decomposition, or  the assumption  in Theorem~\ref{thm:numbers} that $g+4r-2d \neq 4$, which is equivalent to $(H-2L)^2\neq -2$, is not satisfied. Thus \eqref{eq:cliff} is the optimal  bound to be able to apply Theorem~\ref{thm:numbers}.	
\end{remark}

\begin{remark} \label{rem:UD1}
	Condition \eqref{eq:caz} is automatically fulfilled if $r=1$. Moreover,  under assumption \eqref{eq:cliff}, condition \eqref{eq:caz} is for instance implied by the simplified assumption
	\begin{equation} \label{eq:3dc-caz}
		g\geq \min\left\{ r(r+1), \; 10(r-1)\right\} =
		\begin{cases}
			r(r+1), & \mbox{if} \;\; 2 \leq r \leq 7,\\
			& \\
			10(r-1), & \mbox{if} \;\; r \geq 8.
		\end{cases}
	\end{equation}
	Indeed, because of \eqref{eq:cliff}, condition \eqref{eq:caz} is implied by
	\[
	(r-1)(3g-4)^2<2(g-2)\left(\frac{g-3}{2}+2r\right)^2,
	\]
	equivalently  
	\begin{equation} \label{eq:RM2}
		g^3+g^2(-10r+10)+g(16r^2+8r-27)-32r^2+16r+14>0.
	\end{equation} As $g(16r^2+8r-27)-32r^2+16r+14>0$ for $r\geq 2$ and $g\geq 3$, we see that \eqref{eq:RM2} holds if $g\geq 10(r-1)$. If $r \leq 7$, one can check that \eqref{eq:RM2} holds when $g \geq r(r+1)$. Hence  \eqref{eq:caz} holds.
\end{remark}

\begin{proof}[Proof of Proposition~\ref{prop:UD1}]
	We want  to prove that $H  \sim  L+(H-L)$ is the unique  flexible  decomposition of $H$. 
	To this aim, let $H \sim A+B$ where $A=xH-yL$ and $B=(1-x)H+yL$ be a  flexible  decomposition of $H$. We  will  prove that $(x,y)=(1,1)$ or $(0,-1)$, which both correspond to the decomposition $H \sim (H-L)+L$.

	If $r=1$, assumption \eqref{eq:cliff} and Lemma~\ref{lem:ldecPB} show that the only possibilities are $(x,y)=(1,1)$, $(0,-1)$ or $(0,-2)$. In the latter case we have $B \sim H-2L$ and $B^2=-2$. We claim that $B$ is the only $(-2)$-divisor on $S$, which will prove that $B$ is irreducible, giving the contradiction $h^0(B)=1$. Indeed, if $\Gamma$ is a $(-2)$-divisor on $S$, we may write $\Gamma  \sim  aH-bL$, with $a>0$, since multiples of $L$ can never contain $H$. Then $-1=\frac{1}{2}\Gamma^2=a[a(g-1)-bd]$, whence $a=1$ and $g=bd$, so $b=2$ and $\Gamma  \sim  H-2L  \sim  B$. This gives the desired contradiction, showing that the case $(x,y)=(0,-2)$ does not occur, as desired.

	Now suppose $r\geq 2$.
	
	We first claim that if $-3\leq y \leq 3$ and $0\leq x \leq 1$, then again $(x,y)=(1,1)$ or $(0,-1)$. Indeed, Lemma~\ref{lem:ldecPB} shows that it suffices to rule out the cases $(x,y)=(0,- 2), (0,-3), (1, 2)$ and $(1,3)$. In these cases one of $A$ or $B$ is $H-2L$ or $H-3L$. By \eqref{eq:cliff} we have \[(H-2L)^2=2g+8r-10-4d\leq 2g+8r-10-4\left(\frac{g-3}{2}+2r\right)=-4,\] a contradiction. Similarly, we have
	\[-2 \leq (H-3L)^2=2g+18r-20-6d \leq 2g+18r-20-6\left(\frac{g-3}{2}+2r\right)
	= 6r-11-g, \]
	whence $g \leq 6r-9$. But then, using this and \eqref{eq:cliff} once more, we obtain the contradiction
	\[ (H-3L) \cdot H= 2g-2-3d \leq 2g-2-3\left(\frac{g-3}{2}+2r\right)=\frac{g+5}{2}-6r \leq \frac{1}{2}\left(6r-4\right)-6r =-3r-2<0.\]
	
	It remains to show that $-3 \leq y \leq 3$ and $0\leq x \leq 1$. 
	By Remark~\ref{obs:dircomp} and Lemma~\ref{lem:ldecPB}, to show that $0 \leq x \leq 1$ it suffices to show that		
	\begin{equation}\label{eq:cliff1} 
		1+\frac{d}{\sqrt{\abs{\Delta}}\sqrt{g-1}}<2
	\end{equation}
	and
	\begin{equation}\label{eq:cliff2}
		\frac{2(r-1)g}{(d-\sqrt{\abs{\Delta}})\sqrt{\abs{\Delta}}}<2.
	\end{equation}
	Having done this, the case $x=1$ yields, because of the right inequality in \eqref{eq:ABH}, that
	\[ 0< y < \frac{2(g-1)}{d} \leq \frac{2(g-1)}{\frac{g-3}{2}+2r} \leq \frac{4(g-1)}{g+5} <4,\]
	and similarly the case $x=0$ yields, because of the left inequality in \eqref{eq:ABH}, that
	$0< -y <4$. We have therefore left to prove  \eqref{eq:cliff1} and  \eqref{eq:cliff2}.
	
	Inequality \eqref{eq:cliff2} is equivalent to 
	\[
	d^2-(r-1)(3g-4)<d\sqrt{\abs{\Delta}}. \]
	The latter is implied by the inequality obtained by squaring both sides, which is \eqref{eq:caz}. Thus, \eqref{eq:cliff2} is satisfied.
	
	Inequality \eqref{eq:cliff1} is equivalent to 
	\begin{equation} \label{eq:cliff1'}
		4(g-1)^2(r-1)<(g-2)d^2. 
	\end{equation}
	Our assumptions yield that $g \geq 4r-1 \geq 7$, cf. Remark~\ref{rem:g>6}, and then one readily checks that $4(g-1)^2 \leq \frac{1}{2}\left(3g-4\right)^2$, so that \eqref{eq:cliff1'} is implied by \eqref{eq:caz}. Thus, \eqref{eq:cliff1} is also satisfied.
\end{proof}

\section{Bounds on the limits of Brill--Noether special divisors}\label{sec:numbers}

The aim of this section is to give sufficient conditions so that \eqref{eq:condi1}-\eqref{eq:condi4} in Theorem~\ref{thm:numbers} cannot be verified, thus ruling out the existence of a $g^{r'}_{d'}$ as in that theorem.

\begin{lemma}\label{lemma:numcond0}
	Assume that $r'\geq 1$,  $2 \leq d' \leq g-1$,  \eqref{eq:basiccond} is satisfied, and that
	\begin{equation}\label{eq:NC30}
		\begin{cases}
			d-r+r'-d'   >  \frac{r}{r'}-1, & \mbox{if} \;\; r'\leq r, \\
			g+r'-1-d'-\frac{g+r-d-1}{r'-r+1}>0, & \mbox{if} \;\;  r < r'.\end{cases}
	\end{equation}
	Then there are no non-negative integers $r_1,r_2,d_1,d_2$ satisfying \eqref{eq:condi1}--\eqref{eq:condi4}.
\end{lemma}

\begin{proof} Assume, to get a contradiction, that $r_1,r_2,d_1,d_2$ are non-negative integers satisfying \eqref{eq:condi1}--\eqref{eq:condi4}.
	
	We have
	\begin{equation} \label{eq:condi1+}
		0 \leq r_1 \leq r'-1. 
	\end{equation}
	Condition \eqref{eq:condi3} can be rewritten as
	\begin{equation}
		\label{eq:condi3+}
		r_1+\frac{r_1r}{r_1+1} \leq d_1 < r_1+r.
	\end{equation}
	From the left hand inequality of \eqref{eq:condi4} combined with \eqref{eq:condi1} and \eqref{eq:condi2}, along with the fact that $\rho(g,r,-)$ is an increasing function, we obtain
	\[ \rho(g+r-d-1,r'-1-r_1,d'-d+2r-2-d_1) \geq 0,\]
	which can be rewritten as
	\begin{equation}
		\label{eq:condi4+}
		d_1 \leq d'+r+r_1-r'-g +\frac{g+r-d-1}{r'-r_1}.
	\end{equation}  
	
	Grant for the moment the following:
	
	\begin{claim} \label{cl:NC1}
		If
		\begin{itemize}
			\item $r'\leq r$, or
			\item  $r_1 <r <r'$, 
		\end{itemize}
		then
		\begin{equation}
			\label{eq:NC1}
			d'+r+r_1-r'-g +\frac{g+r-d-1}{r'-r_1} < r_1+\frac{r_1r}{r_1+1}.
		\end{equation}
	\end{claim}
	
	We show how the lemma follows from the claim.
	
	Assume that
	$r' \leq r$. Then Claim~\ref{cl:NC1} shows that \eqref{eq:condi4+} and the left hand inequality of \eqref{eq:condi3+} are incompatible. We are therefore done in this case.
	
	Assume, finally, that $r'>r$. Then Claim~\ref{cl:NC1} shows that \eqref{eq:condi4+} and the left hand inequality of \eqref{eq:condi3+} are compatible only for $r_1>r-1$. This latter inequality is on the other hand equivalent to
	\[r+r_1-\left(r_1+\frac{rr_1}{r_1+1} \right) <1.\]
	Thus there is no integer $d_1$ satisfying \eqref{eq:condi3+}. This concludes the proof of the lemma.

	We have left to prove the claim:

	\begin{proof}[Proof of Claim~\ref{cl:NC1}]
		Since the left hand side of \eqref{eq:NC1} is a convex function of $r_1$ and the right hand side is a concave function of $r_1$, it suffices to prove the inequality for the two boundary values $r_1=0$ and 
		\[
		r_1=\begin{cases} 
			r'-1, & \mbox{if} \;\; r'\leq r, \\
			r-1 & \mbox{if} \;\; r'>r,
		\end{cases}
		\] 
		cf. \eqref{eq:condi1+}, where it reads, respectively,
		\begin{equation}
			\label{eq:NC2}
			d'+r-r'-g +\frac{g+r-d-1}{r'} < 0,
		\end{equation}
		\begin{equation}
			\label{eq:NC3}
			d'-d+2r-2 < r'-1+\frac{(r'-1)r}{r'}, \;\; \mbox{if} \;\; r' \leq r,
		\end{equation}
		and
		\begin{equation}
			\label{eq:NC4}
			d'+1-r'-g +\frac{g+r-d-1}{r'-r+1} < 0, \;\; \mbox{if} \;\; r'>r.
		\end{equation}
		One easily verifies that \eqref{eq:NC3} is equivalent to the upper inequality in \eqref{eq:NC30} and implies \eqref{eq:NC2}. We are therefore done in the case $r' \leq r$.
		
		Assume now that 
		$r'>r$. If $r \geq r'-\frac{g-d-1}{r'+1}$, one verifies that \eqref{eq:NC4} implies \eqref{eq:NC2}. If $r < r'-\frac{g-d-1}{r'+1}$, then $g-d-1<(r'-r)(r'+1)$. Using this together with $d' \leq g-1$, one checks that \eqref{eq:NC2} is satisfied. Thus, \eqref{eq:NC2} is redundant, and we are left with \eqref{eq:NC4}, which can be rewritten as the lower inequality in \eqref{eq:NC30}.  This concludes the proof of the claim.
	\end{proof}                        
	Having proved the claim, the lemma follows. 
\end{proof}

\section{Main result and applications} \label{sec:app}

We summarize the results of the previous sections as the main result of the paper.

\begin{thm} \label{thm:main_K3_0}
	Let $g,r,d,r',d'$ be integers satisfying
	\begin{eqnarray}
		\label{eq:basiccond+}
		& g \geq 3, \;\; r \geq 1, \;\; r' \geq 1, \;\; 2 \leq d \leq g-1, \;\;  2 \leq d' \leq g-1, &  \\
		\label{eq:cliff+} & d  \geq 
		\begin{cases}
			\frac{g-3}{2}+2r, & \mbox{if} \; \;  r\geq 2, \\
			\frac{g}{2}, & \mbox{if} \; \; r=1, \end{cases} \\
		\label{eq:caz+}
		&  (r-1)(3g-4)^2  <  2d^2(g-2), & \\
		& \label{eq:rho'} \rho(g,r',d')<0, & \\
		\label{eq:NC30+}
		&\begin{cases}
			d-r+r'-d'   >  \frac{r}{r'}-1, & \mbox{if} \;\; r'\leq r, \\
			g+r'-1-d'-\frac{g+r-d-1}{r'-r+1}>0, & \mbox{if} \;\;  r < r'.\end{cases} &		\end{eqnarray}
	Then for any K3 surface $(S,H)\in\K^r_{g,d}$ with $\Pic(S)=\Lambda^r_{g,d}$, we have that $|H|$ contains smooth irreducible curves, and for any such $C \in |H|$ we have:
	\begin{enumerate}[label=(\roman*)]
		\item $C$ has genus $g$,
		\item $|\O_C(L)|$ is a base point free complete $g^r_d$, which is very ample if $r \geq 3$,
		\item $C$ carries no $g^{r'}_{d'}$.
	\end{enumerate}
\end{thm}

\begin{proof}
	Condition \eqref{eq:caz+} implies \eqref{eq:defD_grd}, which allows us to construct a K3 surface $(S,H)\in\K^r_{g,d}$ with $\Pic(S)=\Lambda^r_{g,d}$ as in \S\ref{sec:K3-ld}. Conditions \eqref{eq:basiccond+}-\eqref{eq:caz+} are the conditions in Proposition~\ref{prop:UD1}, which guarantee that
	$(S,H)$ satisfies decomposition rigidity. For $r \geq 2$, inequality \eqref{eq:cliff+} can be rewritten as
	$g+4r-2d \leq 3$, so that also 
	the assumption that $g+4r-2d \neq 4$ in Theorem
	\ref{thm:numbers} is  satisfied. Thus, by  the latter, properties $(i)$ and $(ii)$ in the theorem are fulfilled. If $(iii)$ is not fulfilled, then, due to \eqref{eq:rho'}, there would be integers $r_1,r_2,d_1,d_2$ satisfying \eqref{eq:condi1}-\eqref{eq:condi4}. But condition \eqref{eq:NC30+} is \eqref{eq:NC30} of  Lemma~\ref{lemma:numcond0}, which then tells us that such integers cannot exist. Thus, $(iii)$ is fulfilled.
\end{proof}

\begin{remark} \label{rem:UD1+}
	Remark~\ref{rem:UD1} tells us that condition \eqref{eq:caz+} can be replaced by the simplified assumption
	\begin{equation} \label{eq:3dc-caz+}
		g\geq \min\left\{ r(r+1), \; 10(r-1)\right\} =
		\begin{cases}
			r(r+1), & \mbox{if} \;\; 2 \leq r \leq 7,\\
			& \\
			10(r-1), & \mbox{if} \;\; r \geq 8.
		\end{cases}
	\end{equation}
\end{remark}

We will now deduce Theorem~\ref{thm:conj_intro} in the introduction from Theorem~\ref{thm:main_K3_0}. To this aim, we first need two lemmas.

\begin{lemma} \label{lemma:basicML2}
	Assume that the triple $(g,r,d)$ is associated to an expected maximal Brill--Noether locus, $g \geq 3$ and $(g,r,d)\neq(6,2,5)$. Then conditions \eqref{eq:cliff+}-\eqref{eq:caz+} in Theorem~\ref{thm:main_K3_0} are satisfied.
\end{lemma}

\begin{proof}
	If $r=1$, then $d=\left\lceil\frac{g}{2}\right\rceil$ by \eqref{dmaxdef}, so \eqref{eq:cliff+} is satisfied. If $r\geq 2$, we have $g \geq r(r+1)$ by \eqref{eq:EM3}. Remark~\ref{rem:UD1+} yields that \eqref{eq:caz+} is satisfied. Using $(g,r) \neq (6,2)$, one can check that the inequality
	\[ r-1+\left\lceil\frac{gr}{r+1}\right\rceil \geq \frac{g-3}{2}+2r\]
	holds, whence \eqref{dmaxdef} implies that \eqref{eq:cliff+} is satisfied. Hence, \eqref{eq:cliff+}-\eqref{eq:caz+} in Theorem~\ref{thm:main_K3_0} are satisfied.
\end{proof}

\begin{lemma}\label{lemma:basicML3}
	Assume that the triple $(g,r,d)$ is associated to an expected maximal Brill--Noether locus, with $2 \leq d \leq g-1$, and $(r',d')$ a pair of integers such that $r' \geq 1$, 
	$2 \leq d' \leq g-1$,  $\rho(g,r',d')<0$, and either
	\begin{enumerate}[label={\normalfont(\alph*)}]
		\item  $(g,r',d')$ is associated to an expected maximal Brill--Noether locus, $r' \neq r$, or
		\item $r'=r$ and $d'  \leq d-1$, or
		\item $r'=r+1$ and $d' \leq d+1$.
	\end{enumerate}
	Then, except for the cases 
	\begin{equation}
		\label{eq:exceptions}
		(g,r,d,r',d') \neq (6,2,5,1,3),(7,2,6,1,4),(8,1,4,2,7), (9,2,7,1,5),
	\end{equation}
	condition \eqref{eq:NC30+} in Theorem~\ref{thm:main_K3_0} is satisfied. 
\end{lemma}

\begin{proof}
	It is straightforward to check that \eqref{eq:NC30+} is satisfied under assumptions (b) and (c). We are therefore left with case (a), where $r' \neq r$	and $(g,r',d')$ is associated to an expected maximal Brill--Noether locus.

	We first prove the lower inequality in \eqref{eq:NC30+}.  After substituting expressions for $d$ and $d^\prime$ given by \eqref{dmaxdef} the inequality is  equivalent to
	\begin{equation}
		\label{eq:NC4'} \left \lceil \frac{gr'}{r'+1} \right\rceil \left(r'+1-r\right)< \left \lceil \frac{gr}{r+1} \right\rceil +g(r'-r).
	\end{equation}
	The latter is implied by
	\[ \left(\frac{gr'}{r'+1} +1\right) \left(r'+1-r\right) \leq \frac{gr}{r+1}  +g(r'-r),\]
	which is equivalent to
	\begin{equation} \label{eq:heart}
		r'+1-r \leq \frac{gr(r'-r)}{(r+1)(r'+1)}.
	\end{equation}
	Since $g \geq r'(r'+1) \geq (r+1)(r'+1)$ by \eqref{eq:EM3}, the latter is implied by
	\[
	r'+1-r \leq r(r'-r), \]
	which is easily seen to hold as soon as $r \geq 2$. This leaves the case $r=1$, where \eqref{eq:heart} reads
	\[ g \geq \frac{2r'(r'+1)}{r'-1}. \]
	Using the fact that $g \geq r'(r'+1)$ by \eqref{eq:EM3}, we see that the latter is satisfied as long as $r' \geq 3$.  Recalling that $r' \geq 2$ (as $r'>r$),  this leaves the case $(r,r')=(1,2)$, in which case 
	\eqref{eq:NC4'}  reads  
	\begin{equation} \label{eq:atalanta}
		2\ceil{\frac{2g}{3}}< g+\ceil{\frac{g}{2}}.
	\end{equation}
	One easily checks that this is satisfied, since we are assuming $(g,r,d,r',d') \neq (8,1,4,2,7)$.  This  finishes the proof of the  lower  inequality in \eqref{eq:NC30+}.

	We finally prove the upper inequality in \eqref{eq:NC30+}.
	After substituting the expressions for $d$ and $d^\prime$ given by \eqref{dmaxdef}, the inequality is equivalent to
	\begin{equation}
		\label{eq:cl11}
		\left\lceil \frac{r'g}{r'+1} \right\rceil +\frac{r}{r'} < \left\lceil \frac{rg}{r+1} \right\rceil +1.
	\end{equation}
	The latter inequality is implied by
	\[
	\frac{r'g}{r'+1} +\frac{r}{r'} \leq \frac{rg}{r+1},
	\]
	which can be rewritten as
	\[
	\frac{r}{r'} \leq  \frac{g(r-r')}{(r'+1)(r+1)}.
	\]
	By \eqref{eq:EM3} the latter inequality is implied by
	\[ \frac{1}{r'} \leq \frac{r-r'}{r'+1}, \]
	which is equivalent to
	\[\frac{r'+1}{r'} \leq r-r'.\]
	Since $r' <r$, the latter is satisfied unless $r'=r-1$. We have therefore left to prove \eqref{eq:cl11}  when  $r'=r-1$, in which case it reads
	\begin{equation}
		\label{eq:cl12}
		\left\lceil \frac{(r-1)g}{r} \right\rceil +\frac{1}{r-1} < \left\lceil \frac{rg}{r+1} \right\rceil.
	\end{equation}
	Recall that 
	$r \geq 2$ (as $r \geq r'+1$).  If $r \geq 3$, then \eqref{eq:cl12} is implied by
	\[ \frac{(r-1)g}{r} +1 \leq \frac{rg}{r+1},\]
	which can be rewritten as $g \geq r^2+r$, which holds by \eqref{eq:EM3}. Thus, \eqref{eq:cl12} is proved  if $r \geq 3$. If $r=2$, then \eqref{eq:cl12}  reads
	
	\[ \ceil{\frac{g}{2}}+1 < \ceil{\frac{2g}{3}}.\]
	Using our assumptions that $(g,r,d,r',d') \neq (6,2,5,1,3),(7,2,6,1,4),(9,2,7,1,5)$, one verifies that the latter inequality is satisfied.  Thus, \eqref{eq:cl12} is proved  if $r =2$. This finishes  the proof of the upper inequality in \eqref{eq:NC30+}.
\end{proof}

\begin{remark}\label{rem:625}
	While the hypotheses of Theorem
	\ref{thm:main_K3_0} do not hold in the case $(g,r,d)=(6,2,5)$, we
	still have the existence of a polarized K3 surface $(S,H)\in\K^2_{6,5}$ such
	that $|H|$ contains smooth irreducible curves of genus 6 and $|\O_C(L)|$ is
	a base point free very ample complete $g^2_5$. Indeed, the existence of $(S,H)$ with $\Pic(S)=\Lambda^2_{6,5}$ such that the conclusions $(i)$--$(ii)$ of
	Theorem~\ref{thm:main_K3_0} are satisfied follows as above; this has, in fact, already
	been done (in a more general setting) in \cite[Thm. 4.3]{ELMS}. In particular, the curves are isomorphic 
	to smooth plane quintics, whence do not carry $g^1_3$s, so they again do not carry any $g^{r'}_{d'}$ with $(r',d') \neq (2,5)$.  One may also check that $(S,H)$ satisfies decomposition rigidity. 
\end{remark}

Finally, applying all this to the expected maximal Brill--Noether
loci, we can now give a proof of Theorem~\ref{thm:conj_intro}.  

\begin{proof}[Proof of Theorem~\ref{thm:conj_intro}]
	Lemma~\ref{lemma:basicML2} tells us that $(g,r,d)$ satisfies conditions \eqref{eq:basiccond+}-\eqref{eq:caz+} in Theorem~\ref{thm:main_K3_0}.
	In particular, by Proposition~\ref{prop:UD1}, the hypotheses in Theorem~\ref{thm:numbers} are satisfied by any K3 surface $S$ with $\Pic(S)=\Lambda^r_{g,d}$, as in \S\ref{subsec:k3notation}. Let $C \in |H|$ be any smooth curve and assume it carries a $g^{r'}_{d'}$ with $\rho(g,r',d') <0$ and $(r',d') \neq (r,d)$, hence $C$ also carries a $g^{r''}_{d''}$ with $(g,r'',d'')$ associated to an expected maximal Brill--Noether locus, obtained using the trivial containments $\M^r_{g,d} \subset \M^r_{g,d+1}$
	and $\M^r_{g,d} \subset \M^{r-1}_{g,d-1}$, along with
	Serre duality.
	
	Assume that $(r'',d'') \neq (r,d)$. Then,  if  $(g,r,d) \neq (6,2,5)$, Lemma~\ref{lemma:basicML3} tells us that all the remaining conditions in Theorem~\ref{thm:main_K3_0} are satisfied for $(r',d')=(r'',d'')$ (recall that we are assuming $(g,r,d) \neq (7,2,6),(8,1,4),(9,2,7)$). Hence, by the same theorem, we get the desired contradiction, that is, that the 
	$g^{r''}_{d''}$ cannot exist. 
	
	Assume that $(r'',d'') = (r,d)$. Then the $\g{r}{d}$ is obtained from the $\g{r'}{d'}$ by a series of trivial containments and Serre duality, the last of which is Serre duality or one of the trivial containments $\BN{g}{r}{d-1} \subset \BN{g}{r}{d}$ or $\BN{g}{r+1}{d+1}\subset \BN{g}{r}{d}$. However, if Serre duality were the last step, then the $\g{g-d+r-1}{2g-2-d}$ was obtained using trivial containments from the Serre duals of a $\g{r+1}{d+1}$ or a $\g{r}{d-1}$.  Thus, in any case, $C$ carries a $\g{r}{d-1}$ or a $\g{r+1}{d+1}$. Again if $(g,r,d) \neq (6,2,5)$, Lemma~\ref{lemma:basicML3} tells us that all the remaining conditions in Theorem~\ref{thm:main_K3_0} are satisfied for $(r',d')=(r,d-1)$ or $(r+1,d+1)$, and we get a contradiction again.

        The remaining case of $(g,r,d) =(6,2,5)$ has already been
        handled by a direct geometric argument in
        \cite[Proposition~6.1]{AH}.  One can, in fact, also handle
        this case using K3 surfaces, as in Remark~\ref{rem:625}.
\end{proof}

\begin{remark}
Theorem~\ref{thm:main_K3_0} directly shows that the general
curve in the Brill--Noether loci $\BN{7}{1}{4}$, $\BN{8}{2}{7}$, and
$\BN{9}{1}{5}$ is $\g{r}{d}$-general for $(r,d)=(1,4),(2,7),(1,5)$,
respectively, which are the maximal Brill--Noether loci in the exceptional genera, as in Example~\ref{ex: unex cont}. Hence Theorem~\ref{thm:conj_intro} says that any expected maximal
Brill--Noether locus $\BN{g}{r}{d}$ contains a $\g{r}{d}$-general
curve, unless $g=7,8,9$.
\end{remark}

\begin{remark}
	This also answers the question of containments of Brill--Noether loci of the form $\BN{g}{r}{d}\subset \BN{g}{r'}{d'}$ with $\rho(g,r,d)=-2$ and $\rho(g,r',d')=-1$, studied in \cite{AHL,CHOI2022,bigas2023brillnoether}. Indeed, as observed in \cite[Lemma 1.2]{BH_2024_BN_deg}, Brill--Noether loci with $\rho=-1,-2$ are expected maximal. 
\end{remark}

We give a few examples in genus $7,8,9$ of distinguishing Brill--Noether loci using Theorem~\ref{thm:main_K3_0}.

\begin{example}
	In genus $7$, Theorem~\ref{thm:main_K3_0} yields that the locus $\BN{7}{2}{6}$ is not contained in any Brill--Noether locus except $\BN{7}{1}{4}$, which is the unexpected containment of Example~\ref{ex: unex cont}.
	
	In genus $8$, we obtain the non-containment $\BN{8}{1}{4} \nsubseteq \BN{8}{2}{6}$, thus the locus $\BN{8}{1}{4}$ is not contained in any Brill--Noether locus except $\BN{8}{2}{7}$, again an unexpected containment of Example~\ref{ex: unex cont}. The containment $\BN{8}{2}{6} \subset \BN{8}{1}{4}$ is obtained in \cite[Lemma 3.4]{Mu2}.
	
	In genus $9$, we obtain that the locus $\BN{9}{2}{7}$ is not contained in any Brill--Noether locus except $\BN{9}{1}{5}$. 
\end{example}

Theorem~\ref{thm:main_K3_0} also gives non-containments of non-maximal Brill--Noether loci.

\begin{example}			
	In genus $11$, we obtain the non-containments $\BN{11}{3}{10}\nsubseteq \BN{11}{2}{8}$ and $\BN{11}{2}{8} \nsubseteq \BN{11}{3}{10}$.
\end{example}

In fact,
Theorem~\ref{thm:main_K3_0} gives infinitely many non-containments of Brill--Noether loci that are not maximal. As a sample, let us see what happens when we lower the maximal degrees a little. Recall the definition of $\dmax(g,r)$ from \eqref{dmaxdef}.

\begin{prop}\label{prop:nonmax}
	Let $\BN{g}{r}{\dmax(g,r)}$ and $\BN{g}{r'}{\dmax(g,r')}$ be expected maximal loci with $r \geq 2$, $r' \geq 2$ and $r \neq r'$. We have a non-containment $\BN{g}{r}{d}\nsubseteq \BN{g}{r'}{\dmax(g,r')-1}$ in any of the following cases:

	\begin{itemize}
		\item $d \geq \dmax(g,r)-1$ and
		\begin{itemize}
			\item[$\circ$] $r' < r$, $(r,r',g) \neq (3,2,12)$, or
			\item[$\circ$] $(r,r')=(2,3)$, $g \in \{13,14,15,16,18\}$,
		\end{itemize}
		\item $d \geq \dmax(g,r)-2$ and
		$(r,r')=(2,4)$, $g \in \{20,21,22,24\}$,
		\item $d \geq \dmax(g,r)-3$ and
		$(r,r')=(2,5)$, $g=30$,
		\item $d \geq \dmax(g,r)-(r'-r+1)$ and
		\begin{itemize}
			\item[$\circ$] $r' > r \geq 3$, or
			\item[$\circ$] $r' \geq 6$, $r=2$, or
			\item[$\circ$] $(r,r')=(2,5)$ and $g \geq 31$, or
			\item[$\circ$] $(r,r')=(2,4)$ and $g \geq 25$, or $g=23$, or
			\item[$\circ$] $(r,r')=(2,3)$ and $g \geq 19$, or $g=17$.
		\end{itemize}
	\end{itemize}
\end{prop}

\begin{proof}
	We have proved that\eqref{eq:NC30+} holds for $d=\dmax(g,r)$ and for $d'=\dmax(g,r')$. Setting instead $d'=\dmax(g,r')-1$, we therefore see that the upper inequality in \eqref{eq:NC30+} holds for  $d\geq \dmax(g,r')-1$, whereas the  lower one  hold for $d \geq \dmax(g,r)-(r'-r+1)$.   Since $g \geq r(r+1)$ by maximality of $\BN{g}{r}{\dmax(g,r)}$ (cf. \eqref{eq:EM3}), Remark~\ref{rem:UD1+} tells us that \eqref{eq:caz+} is redundant. Hence, it remains to check \eqref{eq:cliff+} by substituting for the values of $d$ and using $g \geq r(r+1)$ and $g \geq r'(r'+1)$. We leave the details to the reader.
\end{proof}

\begin{example}			
	In genus $14$, Proposition~\ref{prop:nonmax} gives the non-containments $\BN{14}{3}{12} \nsubseteq \BN{14}{2}{10}$ and $\BN{14}{2}{10} \nsubseteq \BN{14}{3}{12}$.
\end{example}

We also give an example where Theorem~\ref{thm:main_K3_0} does not suffice to prove a non-containment.

\begin{example}
	Let $(g,r,d)=(9,2,6)$ and $(g,r',d')=(9, 3, 8)$.  We  cannot apply Theorem~\ref{thm:numbers}, as on the lattice $\Lambda^2_{9,6}$, we have $(H-2L)^2=0$, hence  $(S,H)$ does not satisfy decomposition rigidity. 
	However, from \cite[Proposition 2.5]{AHL}, we obtain $\BN{9}{2}{6}\nsubseteq \BN{9}{3}{8}$. 
\end{example}

\end{document}